\documentclass{article}
\usepackage{amsmath}
\usepackage{amssymb}
\usepackage{latexsym}
\usepackage{amsthm}
\usepackage{yhmath}
\usepackage{pgfplots}
\usepackage{tikz}
\usepackage{cite}
{\LARGE }

\begin{document}

\baselineskip=16pt

\title{Long-time asymptotic behavior for  the discrete defocusing   mKdV equation}
\author{Meisen Chen,   Engui Fan\thanks{Corresponding author: faneg@fudan.edu.cn}\\[4pt]
 School of Mathematical Sciences, Fudan University,\\
  Shanghai 200433, P.R. China}
\maketitle

\begin{abstract}
In this article, we apply Deift-Zhou nonlin{\footnotesize {\tiny }}ear steepest descent method  to  analyze   the long-time asymptotic behavior of the solution  for
the  discrete defocusing mKdV equation
\begin{equation}
\dot q_n = \left(1-q_n^2\right)\left(q_{n+1}-q_{n-1}\right)  \nonumber
\end{equation}
with   decay initial value
\begin{equation}\nonumber
q_n(t=0) = q_n(0),
\end{equation}
where $n=0,\pm1,\pm2,\cdots$  is a  discrete variable and $t$ is continuous time variable.  This    equation  was proposed by Ablowitz and Ladik.\\[3pt]
{\bf Key words:}  discrete defocusing mKdV equation, Lax pair, Riemann-Hilbert problem,
Deift-Zhou  steepest descent method, long-time asymptotic behavior.

\end{abstract}

{\normalsize }

\section{Introduction}\label{S:int}
\indent

Since Deift and Zhou developed nonlinear steepest descent method in 1993 \cite{deift1993steepest},  it  has been successfully applied    to   analyze   the long-time asymptotic behavior
of a wide variety of continuous integrable systems, such as the mKdV equation,  the  NLS equation, the sine-Gordon equation,  the KdV equation, and the Cammasa-Holm equation \cite{deift1993long,cheng1999long,vartanian2000higher,grunert2009long,de2009long}.
However, there still has been little work on the long-time behavior of the discrete integrable systems, except for the Toda lattice and discrete NLS equation \cite{kruger2009long,yamane2014long}.

In this article, we consider  the following  discrete defocusing  mKdV equation
\begin{equation}\label{eq:1}
\dot q_n = \left(1-q_n^2\right)\left(q_{n+1}-q_{n-1}\right)
\end{equation}
with initial value
\begin{equation}\label{eq:2}
q_n(t=0) = q_n(0),
\end{equation}
where $\dot q_n= dq_n(t)/dt, \ n=0,\pm1,\pm2,\cdots$ and $t $ is continuous time variable. The  equation (\ref{eq:1})
was  first proposed by  Ablowitz and Ladik  \cite{ablowitz1977nonlinear}.
They further showed   that the  equation (1) is the semi-discrete version of the following  classical mKdV equation \cite{ablowitz1977nonlinear}
\[
u_{\tau}(x,\tau)+6u^2u_x(x,\tau)-u_{xxx}(x,\tau)=0.
\]
Narita  found    two kinds of   Miura transformations
\begin{align*}
q_n&=\frac{(u_{n-1}+u_{n+1})u_n-2u_{n-1}u_{n+1}}{(u_{n-1}-u_{n+1})u_n}, \\
q_n&=\frac{u_{n-1}-2u_n+u_{n+1}}{u_{n-1}-u_{n+1}}
\end{align*}
between the   equation (\ref{eq:1}) and  Sokolov-Shabat  equation \cite{narita1997miura}
$$\dot u_n=4(u_{n-1}-u_n)(u_n-u_{n+1})/(u_{n-1}-u_{n+1}).$$
The  some kinds of exact  solutions   of equation  (\ref{eq:1}) were  obtained  by using homotopy analysis method and exp-function method  \cite{zhen2008solitary,zhu2007exp}.
With Darboux transformation,  Wen and Gao obtained the explicit solutions for the equation (\ref{eq:1}) on the discrete spectral of Lax pair \cite{xiao2010darboux}.
However, the solutions of the equation (\ref{eq:1}) with  initial-boundary condition on the continuous spectral  have
been  still unknown  by using inverse scattering transformation or   Riemann-Hilbert approach.
 So in this paper, we would like to apply Riemann-Hilbert approach/Deift-Zhou nonlinear steepest descent method to investigate the long-time behavior of the solution for
the initial value problem of discrete defocusing mKdV equation  (\ref{eq:1})-(\ref{eq:2}).

 The organization of this paper is as follows.
In section 2,   we  introduce   appropriate eigenfunctions and spectral functions  to  reformulate
initial value problem of  discrete  mKdV equation  (\ref{eq:1})-(\ref{eq:2}) as a  Riemann-Hilbert problem(RHP).
From section 3 to section 8,  we transform the RHP to a  model one by using a series of deformations and decompositions.
In section 9, we show the existence and boundedness of the Cauchy operators  during  solving the RHP.  At  the last section, we  provide   the   asymptotic behavior
of the solution of   the    equation (\ref{eq:1}).

\section{Riemann-Hilbert problem}\label{S:inv}
\indent

In this section,  we first investigate the solvability of  the initial value problem  (\ref{eq:1})-(\ref{eq:2}) equation   and then   transform
it into a RHP.    Moreover,  we further  express the potential  $q_n$  of  the discrete defocusing  mKdV equation (\ref{eq:1})  with  a  solution of  the  obtained  RHP.

 Similar  to Proposition 2.1 in \cite{yamane2014long},  we can show the
 following proposition.
\newtheorem{proposition}{Proposition}[section]\label{pro:2.1}
 \begin{proposition}
 Let  $s$  be  a nonnegative integer. If  the initial condition  (\ref{eq:2})
 satisfies
  \begin{align}
  \left\|q_n(0)\right\|_{1,s} &= \sum_{n=-\infty}^{\infty}{(1+\left
  |n\right|)^s\left|q_n(0)\right|} < \infty, \label{eq:3}\\
  \left\|q_n(0)\right\|_\infty &= \sup_n\left|q_n(0)\right| < 1,   \label{eq:4}
  \end{align}
  then the  equation  (\ref{eq:1})  admits an unique solution  in the space
  $$\textit{l}^{1,s}=\{\{c_n\}_{n=-\infty}^\infty :
  \sum_{n=-\infty}^{\infty}{\left(1+\left |n\right|\right)^s\left|c_n\right|}< \infty, \ 0\le t<\infty  \}.$$
 \end{proposition}

 \begin{proof}
   We define
  $$c_{-\infty}=\prod_{n=-\infty}^{\infty}\left(1-\left|q_n\right|^2\right),\ \ \rho_0=\left(1-c_{-\infty}\right)^{\frac{1}{2}},$$
then  we  have
  \begin{equation}\label{eq:5s}
  \sup_n|q_n(0)|\leq \rho_0.
  \end{equation}
  It is easy to verify that $c_{-\infty}$  and $\rho_0$ are  both  conserved quantities.
   Therefore,  we consider equation
  (\ref{eq:1}) as an ordinary differential  equation with respect to $t$,   whose  solution belongs  to  the
  Banach space $\textit{l}^{1,s}\subset\textit{l}^\infty$ under the  condition
  (\ref{eq:3}) and (\ref{eq:4}).

 Given the ball in $l^{\infty}$
 $$B:=\{\{q_n(t)\}\in\textit{l}^\infty, \ |q_n(t)-q_n(0)|<\rho_0\},$$
  we show that  the equation (\ref{eq:1}) admits a solution  in the space  $B$.
  Since the right-hand side of (\ref{eq:1}) is Lipschitz continuous
  and bounded, there exits a $t_1$ such that equation (\ref{eq:1}) admits solution in $B$ for $t \in \left(0,t_1\right]$. By the
  standard argument about ordinary differential  equations  in \cite{lang2012differential}, we get that $t_1$ depends on $\rho_0$   so as we have (\ref{eq:5s}).
   Because we have known that $\rho_0$ is conserved, we have
  \begin{equation}
\sup_n|q_n(t_1)|\leq \rho_0.
\end{equation}
 Similarly, we could extend the solution of equation (\ref{eq:1})
   to $t \in \left(t_1,2t_1\right]$.  Repeating the
  procedure above,  we then can extend the solution to $t\in\left[0,\infty\right)$ in the
  space $\textit{l}^\infty$.  Therefore,  for $0 \le t < \infty$, we have
  $$\sup_n|q_n(t)|\leq \rho_0.$$

 Based on the fact that  there is a solution for (\ref{eq:1})-(\ref{eq:2}) in $l^\infty$ shown above,
 we further verify that the solution belongs to the Banach space $l^{1,s}$. From equation (\ref{eq:1}), we get
  $$\left\|\dot q_n(t)\right\|_{1,s} \le Const.\left\|q_n(t)\right\|_{1,s}.$$
   By integrating $\dot q_n(t)$ with respect to $t$,  we get
  $$\left\|q_n(t)\right\|_{1,s} \le \left\|q_n(0)\right\|_{1,s}+Const.\int_{0}^{t}\left\|q_n(\tau)\right\|_{1,s}d\tau.$$
  By virtue of the Gronwall inequality, it follows that
  $\left\|q_n(t)\right\|_{1,s}$ grows  at most exponentially and does
  not blow up.

  For the uniqueness, if we set $$q_n(0)=0, $$ then the problem (\ref{eq:1})-(\ref{eq:2}) only admits a zero solution.
   \end{proof}

In the following, we transform the initial problem   (\ref{eq:1})-(\ref{eq:2})   into   a  RHP.
It is known that the discrete mKdV equation  admits a Lax pair \cite{ablowitz1977nonlinear}
\begin{eqnarray}
 X_{n+1} &=& z^{\sigma_3}X_n +Q_nX_n,  \label{eq:5}\\
 \dot X_n &=& z^{2\sigma_3}X_n+B_nX_n,    \label{eq:6}
\end{eqnarray}
where $X_n$ is a $2\times 2$ matrix, and
\begin{eqnarray*}
 \sigma_3= \left(
 \begin{matrix}
 1 & 0\\
 0 & -1
 \end{matrix}
 \right), \ \    z^{\sigma_3}= \left(
 \begin{matrix}
 z & 0\\
 0 & z^{-1}
 \end{matrix}
 \right),\ \  Q_n = \left(
 \begin{matrix}
 0 & q_n\\
 q_n & 0
 \end{matrix}
 \right), \\
 B_n = \left(
 \begin{matrix}
 -q_{n-1}q_n & q_nz+q_{n-1}z^{-1}\\
 q_nz^{-1}+q_{n-1}z & -q_{n-1}q_n
 \end{matrix}
 \right).
\end{eqnarray*}

The Lax pair (\ref{eq:5})-(\ref{eq:6}) admit the following asymptotic Jost  solutions
 $$X_n(z,t)\thicksim z^{n\sigma_3}\left(
\begin{matrix}
e^{z^2t} & 0 \\
0 &e^{z^{-2}t}
\end{matrix}\right), \ \ n\to\pm\infty.$$
Making transformation
 $$Y_n=z^{-n\sigma_3}\left(
 \begin{matrix}
  e^{-z^2t} & 0 \\
  0 &e^{-z^{-2}t}
 \end{matrix}\right)X_n,$$
 then we have
\begin{equation}
 Y_{n+1}-Y_n=z^{-\sigma_3}\tilde Q_nY_n, \label{eq:7}
\end{equation}
where
\begin{align}
\tilde Q_n&=z^{-n\hat \sigma_3}e^{-\frac{t}{2}(z^2-z^{-2})\hat\sigma_3}Q_n\notag \\
&=\left(\begin{matrix}
0 & q_nz^{-2n}e^{-t(z^2-z^{-2})}\\ q_nz^{2n}e^{t(z^2-z^{-2})} & 0
\end{matrix}\right),\label{eq:10s}
\end{align}
and for   a $2\times2$ matrix   $A$,  the  symbol  $e^{\hat\sigma_3}$  is defined by
\begin{align*}
e^{\hat\sigma_3}A& \equiv  e^{\sigma_3}Ae^{-\sigma_3}.
\end{align*}

Denoting  $Y_n^{(\pm)}$  as  the $2\times2$ eigenfunctions of (9)  such that
\begin{align}
Y_n^{(-)}\to I & \text{ as } n\to -\infty,\label{eq:8} \\
Y_n^{(+)}\to I & \text{ as } n\to +\infty,\label{eq:9}
\end{align}
 then  we have
\begin{align}
Y_n^{(-)}&=I+\sum_{k=-\infty}^{n-1}z^{-\sigma_3}\tilde Q_kY_k^{(-)}\label{eq:13s},\\
Y_n^{(+)}&=I-\sum_{k=n}^\infty z^{-\sigma_3}\tilde Q_kY_k^{(+)} \label{eq:14s}.
\end{align}
By WKB expansion method, it follows that
\begin{align}
Y_n^{(-)}=&z^{-n\hat\sigma_3}e^{\frac{t}{2}(z^2-z^{-2})\hat\sigma_3}\notag\\&\times\left(
\begin{matrix}
1+O(z^{-2},even) & zq_{n-1}+O(z^3, odd)\\
z^{-1}q_{n-1}+O(z^{-3},odd) & 1+O(z^2,even)
\end{matrix}\right),\label{eq:15}\\
Y_n^{(+)}=&z^{-n\hat\sigma_3}e^{\frac{t}{2}(z^2-z^{-2})\hat\sigma_3}\notag\\&\times\left(
\begin{matrix}
c_n^{-1}+O(z^2,even) & -c_n^{-1}z^{-1}q_n+O(z^{-3}, odd)\\
-c_n^{-1}zq_n+O(z^3,odd) & c_n^{-1}+O(z^{-2},even)
\end{matrix}\right),\label{eq:16}
\end{align}
where $O(z^{\pm3},odd)$ ($O(z^{\pm2},even)$) means the remaining part containing $z^{\pm3}, z^{\pm5},\cdots$ ($z^{\pm 2},z^{\pm4},\cdots$, respectively) and
$$c_n=\prod_{k=n}^{\infty}\left(1-\left|q_k\right|^2\right).$$

Rewriting
\begin{align*}
Y_n^{(-)} =(Y_{n,1}^{(-)},Y_{n,2}^{(-)}), \ \  Y_n^{(+)} =(Y_{n,1}^{(+)},Y_{n,2}^{(+)}),
\end{align*}
where  $Y_{n,j}^{(\pm)}$ is the $j$th column of $Y_n^{(\pm)}$,  and letting
\begin{equation}\label{eq:17s}
X_n^{(\pm)}=z^{n\sigma_3}\left(
\begin{matrix}
e^{z^2t} & 0 \\
0 &e^{z^{-2}t}
\end{matrix}\right)Y_n^{(\pm)},
\end{equation}
from (\ref{eq:8}) and (\ref{eq:9}),   we find that  (\ref{eq:5}) admits matrix solutions $X_n^{(-)}=(X_{n,1}^{(-)},X_{n,2}^{(-)})$ and $X_n^{(+)}=(X_{n,1}^{(+)},X_{n,2}^{(+)})$ such that
\begin{align}
X_{n,1}^{(-)} &\to z^ne^{z^2t}\left(
\begin{array}{c}
1 \\ 0
\end{array}\right),&
X_{n,2}^{(-)} &\to z^{-n}e^{z^{-2}t}\left(
\begin{array}{c}
0 \\ 1
\end{array}\right),&
as\  \  n &\to -\infty,  \label{eq:10}\\
X_{n,1}^{(+)} &\to z^ne^{z^2t}\left(
\begin{array}{c}
1 \\ 0
\end{array}\right),&
X_{n,2}^{(+)} &\to z^{-n}e^{z^{-2}t}\left(
\begin{array}{c}
0 \\ 1
\end{array}\right),&
as \  \ n &\to \infty. \label{eq:11}
\end{align}

From  spectral problem   (\ref{eq:5}),    we know  that $$\det(X_n^{(\pm)})\neq 0,$$
which implies that  $X_n^{(\pm)}$ are invertible.  Thus,
by the linearity of the eigenfunction,   there
exist  four functions $a(z), \ b(z),\  a^*(z)$ and $b^*(z)$,  such
that
\begin{align}
X_{n,1}^{(-)}(z,t)&=a(z)X_{n,1}^{(+)}(z,t)+b(z)X_{n,2}^{(+)}(z,t), \label{eq:17} \\
X_{n,2}^{(-)}(z,t)&=b^*(z)X_{n,1}^{(+)}(z,t)+a^*(z)X_{n,2}^{(+)}(z,t),\label{eq:18}
\end{align}
  which  combining with  (\ref{eq:5}),  we can  get the symmetry
$$a^*(z)=\overline{a(\bar z^{-1})},\ \ b^*(z)=\overline{b(\bar z^{-1})}.$$
By using  (\ref{eq:17s}) and (\ref{eq:17}),  we find  that
\begin{align*}
a(z)&=   \frac{\det(X_{n,1}^{(-)},X_{n,2}^{(+)})}{   \det X_n^{(+)}}  = \frac{\det(Y_{n,1}^{(-)},Y_{n,2}^{(+)})}{   \det Y_n^{(+)}},\\
b(z)&=   \frac{ \det(X_{n,1}^{(+)},X_{n,1}^{(-)}} {\det X_n^{(+)}}    = \frac{ \det(Y_{n,1}^{(+)},Y_{n,1}^{(-)}} {\det Y_n^{(+)}} .
\end{align*}

From  (\ref{eq:10s}), (\ref{eq:13s}) and (\ref{eq:14s}), it is worthy of notice that both $Y_{n,1}^{(-)}$ and $Y_{n,2}^{(+)}$ belong to
$C[\{\left|z\right|\geq 1\}\cup\{\infty\}]$ and they are both
analytic outside the unit circle.  While both $Y_{n,2}^{(-)}$ and $Y_{n,1}^{(+)}$ are continuous on the closed unit disc and
analytic inside the unit circle.    Thus,    using (\ref{eq:15})-(\ref{eq:16}),  we know that  $a(z)$ is analytic outside the unit circle, moreover $a\to 1 \text{ as }z\to \infty.$

Besides, $b$ is continuous on the unit circle, and
\begin{equation}\label{eq:19b}
a^*(0)=1.
\end{equation}
By calculating the determinant of $X_n^{(-)}$ and  $X_n^{(+)}$, we obtain that $$|a(z)|^2-|b(z)|^2=c_{-\infty}>0, \ \ {\rm for}\  |z|=1,$$
which implies $a\ne 0$ on the unit circle.
Let  $r=b/a,$  $\bar r=b^*/a^*$,  we then  have
$$\bar r(z)=\overline{r(\bar z^{-1})},\ \ 0\leq|r(z)|<1.$$

We  now construct the RHP for equation  (\ref{eq:1}).
Defining  a $2\times 2$  analytic matrix function on $\mathbb{C}\setminus\{\left|z\right|=1\}$
\begin{equation}\label{eq:19a}
m(z;n,t)=
\begin{cases}
\left(
\begin{matrix}
1 & 0  \\
0 & c_n
\end{matrix}\right)z^{n\hat\sigma_3}e^{-\frac{t}{2}(z^2-z^{-2})\hat\sigma_3}\left(\frac{Y_{n,1}^{(-)}}{a}, Y_{n,2}^{(+)}\right) & \left|z\right|> 1, \\
\left(
\begin{matrix}
1 & 0  \\
0 & c_n
\end{matrix}\right)z^{n\hat\sigma_3}e^{-\frac{t}{2}(z^2-z^{-2})\hat\sigma_3}\left(Y_{n,1}^{(+)}, \frac{Y_{n,2}^{(-)}}{a^*}\right) & \left|z\right|< 1,
\end{cases}
\end{equation}
from   (\ref{eq:15})-(\ref{eq:16}),   we  can    get   the expansion of  $m$
\begin{equation}\label{eq:21s}
m=I+m_1z^{-1}+\cdots, \ z\rightarrow \infty.
\end{equation}
 Moreover, with (\ref{eq:15}), (\ref{eq:19b}) and (\ref{eq:19a}), we  find  that
\begin{equation}\label{eq:22}
q_n=\frac{m_{12}}{z}\Big|_{z=0}=\frac{d}{dz}m_{12}\Big|_{z=0},
\end{equation}
where $m_{12}$ is the (1,2)-entry of the matrix $m$.

By using  (\ref{eq:17}) and (\ref{eq:18}),  we  can derive that
\begin{equation}\label{eq:12}
 m^+=m^-v, \ \ \   z\in \Sigma,
\end{equation}
where  $\Sigma=\{z:\ \left|z\right|=1\}$ is called jump curve (see Figure 1),  and
\begin{equation}
 v(z;n,t)=z^{n\hat\sigma_3}e^{-\frac{t}{2}(z^2-z^{-2})\hat\sigma_3}\left(
  \begin{matrix}
   1-\left|r(z)\right|^2&-\bar r(z)  \\
   r(z) & 1
  \end{matrix}\right),\nonumber
\end{equation}
 is called the jump matrix.

  In summary, we have got the RHP of $m$
  \begin{align*}
  \begin{cases}
  m \  {\rm analytic\ on } & \mathbb{C} \setminus\{\left|z\right|=1\},\\
  m^+=m^-v,  &z\in\Sigma,\\
  m(z;n,t)\to I &as\  z\to\infty.
  \end{cases}
  \end{align*}
Letting
\begin{equation}
\varphi=\frac{t}{2}(z^2-z^{-2})-n\log z,\label{eq:27}
\end{equation}
 we obtain that
\begin{equation}
  v(z;n,t)=e^{-\varphi\hat\sigma_3}\left(
  \begin{matrix}
    1-\left|r(z)\right|^2&-\bar r(z) \\
    r(z) & 1
  \end{matrix}\right).
\end{equation}

\ \ \ \ \ \ \ \ \ \ \ \ \ \ \ \ \ \
\begin{tikzpicture}
\pgfplotsset{
every axis/.append style={
	extra description/.code={
		\node at (0.5, 0) { {Figure 1}.  The jump curve \emph{$\Sigma$}; };
		\node at (0.4,0.81) {\scriptsize{+}};
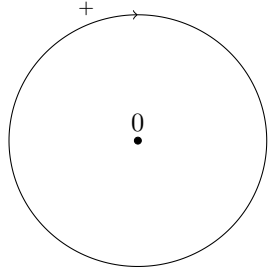
		\node at (0.5,-0.1) {\normalsize{'+' means the orientation of the contour}};
		\filldraw [black] (0.5,0.5) circle (1.3pt) node [above] {0};
	},
},
}
\begin{axis}[
hide x axis,
hide y axis,
ymin=-1.7,
ymax=1.7,
xmin=-2.0,
xmax=2.0,]
\addplot[domain=0:2*pi, ->, samples=200]
({sin(deg(x))},{cos(deg(x))});
\end{axis}
\end{tikzpicture}

\section{ The first RHP deformation  }\label{S:3}
\indent

In this section,   we change the RHP of $m$     into a new equivalent one,  such that its jump matrix admits a helpful lower/upper triangular factorization.
We take a  real constant number   $ 0<V_0 < 2$,  we will discuss the asymptotic behavior under
assumption $\left|n\right| \leq V_0t,  \ t \to \infty$.

From  (\ref{eq:27}),  the function  $\varphi$  has  four  first order stationary phase points
\[
S_1=A,\ S_2=\bar A,\ S_3=-A,\  \ S_4=-\bar A,
\]
where $$A=2^{-1}\left(\sqrt{2+\frac{n}{t}}-i\sqrt{2-\frac{n}{t}}\right).$$

\ \ \ \ \ \ \ \ \ \ \ \ \ \ \ \ \ \ \ \
 \begin{tikzpicture}
  \pgfplotsset{
   every axis/.append style={
   extra description/.code={
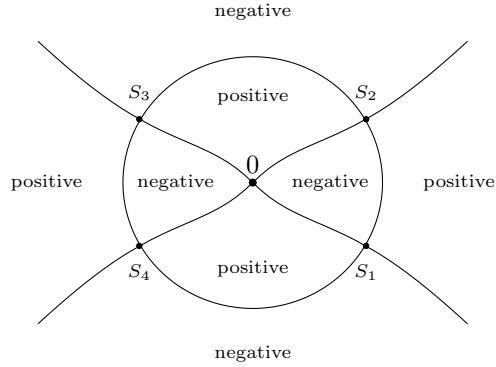
   	\node at (0.5,0) {{Figure 2}. \emph{The sign figure of $Re\varphi$}};
   \node at (0.65,0.5) {\scriptsize{negative}};
   \node at (0.5,0.7){\scriptsize{positive}};
   \node at (0.5,0.3){\scriptsize{positive}};
   \node at (0.35,0.5) {\scriptsize{negative}};
   \node at (0.5,0.9) {\scriptsize{negative}};
   \node at (0.5,0.1) {\scriptsize{negative}};
   \node at (0.1,0.5){\scriptsize{positive}};
   \node at (0.9,0.5){\scriptsize{positive}};
   \node at (0.72,0.29){\scriptsize{$S_1$}};
   \node at (0.72,0.71){\scriptsize{$S_2$}};
   \node at (0.28,0.71){\scriptsize{$S_3$}};
   \node at (0.28,0.29){\scriptsize{$S_4$}};
   \filldraw [black] (0.28,0.352) circle (1pt);
   \filldraw [black] (0.72,0.352) circle (1pt);
   \filldraw [black] (0.28,0.648) circle (1pt);
   \filldraw [black] (0.72,0.648) circle (1pt);
   \filldraw [black] (0.5,0.5) circle (1.3pt) node [above] {0};
   },
   },
   }
   \begin{axis}[
   hide x axis,
   hide y axis,
   ymin=-1.7,
   ymax=1.7,
   ]
   \addplot[domain=0:2]
   ({sqrt(x^2/2+((x^4)*ln(x))/(x^4-1))},{sqrt(x^2/2-(((x^4)*ln(x))/(x^4-1))});
   \addplot[domain=0:2]
   ({-sqrt(x^2/2+((x^4)*ln(x))/(x^4-1))},{-sqrt(x^2/2-(((x^4)*ln(x))/(x^4-1))});
   \addplot[domain=0:2]
   ({-sqrt(x^2/2+((x^4)*ln(x))/(x^4-1))},{sqrt(x^2/2-(((x^4)*ln(x))/(x^4-1))});
   \addplot[domain=0:2]
   ({sqrt(x^2/2+((x^4)*ln(x))/(x^4-1))},{-sqrt(x^2/2-(((x^4)*ln(x))/(x^4-1))});
   \addplot[domain=0:2*pi, samples=200]
   ({sin(deg(x))},{cos(deg(x))});
   \end{axis}
 \end{tikzpicture}

\noindent
Let  $$z=|z|e^{i\theta}.$$ From (\ref{eq:27}),  we have
$${\rm Re}\varphi=\frac{t}{2}\cos(2\theta)(|z|^2-|z|^{-2})-n\log |z|,$$
which leads to  the sign figure of ${\rm Re}\varphi$ as shown in  Figure 2.

We denote $\wideparen{S_jS_{j+1}}$ as the arc on  unit circle from $S_j$ to $S_{j+1}$ such that the central angle of the arc is less than $\pi$ ($S_5=S_1$). Define $\delta$ as an analytic function on $\mathbb{C}\setminus\{|z|=1\}$ satisfying an scalar RHP
\[
\begin{cases}
 \delta_+(z)=\delta_-(z)\left(1-\left|r\right|^2\right) &\ \ \text{on $\wideparen{S_1S_2}\cup \wideparen{S_3S_4}$},\\
 \delta_+(z)=\delta_-(z)  &\ \ \text{on $\wideparen{S_2S_3}\cup \wideparen{S_4S_1}$}, \\
 \delta \to 1             &\ \ as\ z\to\infty,
\end{cases}
\]
 which admits  a unique solution
\begin{equation}
 \delta(z)=e^{-\frac{1}{2\pi i}\left(\int_{S_1}^{S2}+\int_{S_3}^{S_4}\right)\log{\left(1-\left|r(\tau)\right|^2\right)}\frac{d\tau}{\tau-z}},
\end{equation}
where  $\int_{S_j}^{S_{j+1}}$, $j=1, 3$  denote  the integral on $\wideparen{S_jS_{j+1}}$.
Noticing that
$$\sup_{|z|=1}|r(z)|<1,$$
it  can be shown that both $\delta$ and $\delta^{-1}$  are bounded.

We introduce a  transformation on $\mathbb{C}\setminus\{|z|=1\}$
\begin{equation}\label{eq:26}
m^{(1)}=m\delta^{-\sigma_3},
\end{equation}
where  the jump curve still is $\Sigma^{(1)}=\Sigma$, but its orientation  is clockwise on $\wideparen{S_1S_2}\cup\wideparen{S_3S_4}$
and counterclockwise on $\wideparen{S_2S_3}\cup\wideparen{S_4S_1}$  as shown in  Figure 3.

On $\Sigma$,   we find that
\begin{alignat}{2}
 \left(m\delta^{-\sigma_3}\right)_+ &=m_-v\delta_+^{-\sigma_3} =\left(m\delta^{-\sigma_3}\right)_-\delta_-^{\sigma_3}v\delta_+^{-\sigma_3} \label{eq:20},
\end{alignat}
where
\begin{alignat}{2}\label{eq:21}
 \delta_-^{\sigma_3}v\delta_+^{-\sigma_3}
 &=e^{-\varphi\hat\sigma_3}\left(
 \begin{matrix}
  \left(1-\left|r(z)\right|^2\right)\delta_-\delta_+^{-1} & -\bar r(z)\delta_+\delta_- \\
  r(z)\delta_-^{-1}\delta_+^{-1} & \delta_-^{-1}\delta_+
 \end{matrix}\right) \notag\\
 &=
 \begin{cases}
  e^{-\varphi\hat\sigma_3}\left(
  \begin{matrix}
   1-\left|r(z)\right|^2 & -\bar r(z)\delta_+^2 \\
   r(z)\delta_-^{-2} & 1
  \end{matrix}\right), \\
  e^{-\varphi\hat\sigma_3}\left(
  \begin{matrix}
   1 & \frac{-\bar r(z)\delta_+^2}{1-\left|r(z)\right|^2} \\
   \frac{r(z)\delta_-^{-2}}{1-\left|r(z)\right|^2} & 1-\left|r(z)\right|^2
  \end{matrix}\right)
 \end{cases} \notag\\
 &=
 \begin{cases}
  \left(
  \begin{matrix}
   1 & -\bar r(z)\delta_+^2e^{-2\varphi} \\
   0 & 1
  \end{matrix}\right)\left(
  \begin{matrix}
   1 & 0 \\
   r(z)\delta_-^{-2}e^{2\varphi} &1
  \end{matrix}\right)  &\text{on $\wideparen{S_2S_3}\cup \wideparen{S_4S_1}$},\\
  \left(
  \begin{matrix}
   1 &0 \\
   \frac{r\delta_-^{-2}e^{2\varphi}}{1-\left|r\right|^2} & 1
  \end{matrix}\right)\left(
  \begin{matrix}
   1& -\frac{\bar r\delta_+^2e^{-2\varphi}}{1-\left|r\right|^2} \\
   0 &1
  \end{matrix}\right) &\text{on $\wideparen{S_1S_2}\cup \wideparen{S_3S_4}$}.
 \end{cases}\notag
\end{alignat}
Also, we verify that
\begin{equation}
 m^{(1)}=m\delta^{-\sigma_3} \to I \ \  \text{    as $z \to \infty$}. \notag
\end{equation}

So if    setting
\[
v^{(1)}=
 \begin{cases}
   \delta_-^{\sigma_3}v\delta_+^{-\sigma_3} &\text{on $\wideparen{S_1S_2}\cup \wideparen{S_3S_4}$},\\
   \left(\delta_-^{\sigma_3}v\delta_+^{-\sigma_3}\right)^{-1} &\text{on $\wideparen{S_2S_3}\cup \wideparen{S_4S_1}$},
 \end{cases}
\]
then we   get   RHP for $m^{(1)}$   as follows
\begin{equation}\label{eq:19}
 m_+^{(1)}=m_-^{(1)}v^{(1)},  \ \ z\in\Sigma^{(1)}
\end{equation}
To express the jump matrix more concise, we set
\begin{equation}
 \rho(z)=
 \begin{cases}
  -\frac{\bar r}{1-\left|r\right|^2} &\text{on $\wideparen{S_1S_2}\cup \wideparen{S_3S_4}$},\\
  \bar r &\text{on $\wideparen{S_2S_3}\cup \wideparen{S_4S_1}$},
 \end{cases}
\end{equation}
and $$\bar\rho(z)=\overline{\rho(\bar z^{-1})},$$
 then  we have

\begin{equation}\label{eq:34s}
 v^{(1)}=b_-^{-1}b_+,
\end{equation}
where
\begin{equation}\label{eq:35s}
 b_-=\left(
 \begin{matrix}
  1 & 0 \\
  \bar\rho\delta_-^{-2}e^{2\varphi} & 1
 \end{matrix}\right),\ \ \  \ \ \ \
 b_+=\left(
 \begin{matrix}
  1 & \rho\delta_+^2e^{-2\varphi} \\
  0 & 1
 \end{matrix}\right).
\end{equation}

\ \ \ \ \ \ \ \ \ \ \ \ \ \ \ \ \ \
\begin{tikzpicture}
\pgfplotsset{
	every axis/.append style={
		extra description/.code={
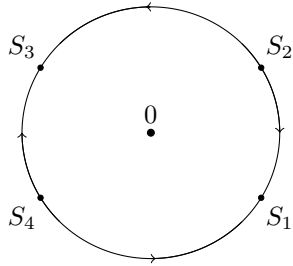
			\node at (0.5,0) {Figure 3.  The jump curve \emph{$\Sigma^{(1)}$}};
			\node at (0.75,0.3) {\normalsize{$S_1$}};
			\node at (0.75,0.7) {\normalsize{$S_2$}};
			\node at (0.25,0.7) {\normalsize{$S_3$}};
			\node at (0.25,0.3) {\normalsize{$S_4$}};
			\filldraw [black] (0.286,0.348) circle (1pt);
			\filldraw [black] (0.714,0.348) circle (1pt);
			\filldraw [black] (0.286,0.652) circle (1pt);
			\filldraw [black] (0.714,0.652) circle (1pt);
			\filldraw [black] (0.5,0.5) circle (1.3pt) node [above] {0};
		},
	},
}
\begin{axis}[
hide x axis,
hide y axis,
ymin=-1.7,
ymax=1.7,
xmin=-2.0,
xmax=2.0,]
\addplot[domain=0:2*pi, samples=200]
({sin(deg(x))},{cos(deg(x))});
\addplot[domain=1.25*pi:1.5*pi, ->, samples=200]
({sin(deg(x))},{cos(deg(x))});
\addplot[domain=0.25*pi:0.5*pi, ->, samples=200]
({sin(deg(x))},{cos(deg(x))});
\addplot[domain=1.75*pi:2*pi, -<, samples=200]
({sin(deg(x))},{cos(deg(x))});
\addplot[domain=0.75*pi:pi, -<, samples=200]
({sin(deg(x))},{cos(deg(x))});
\end{axis}
\end{tikzpicture}

\section{The second RHP deformation } \label{S:dec}
\indent

In this section, we would decompose $\rho$ and $\bar\rho$  into the rational part and two small parts respectively,
 so that we can  make the second RHP transformation with augmented jump contour (see  {Figure 4}).

\ \ \ \ \ \ \ \ \ \ \ \ \ \
\begin{tikzpicture}
\pgfplotsset{
	every axis/.append style={
		extra description/.code={
			\node at (0.5,0.02) { {Figure 4.} \ \  Jump curve  $\Sigma^{(2)}$};
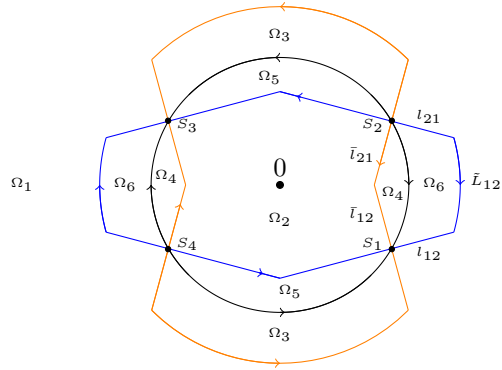
			\node at (0.5,-0.05) {\emph{$L$ is in blue and $\bar L$ is in orange}};
			\node at (0.5,0.42){\tiny{$\Omega_2$}};
			\node at (0,0.5){\tiny{$\Omega_1$}};
			\node at (0.5,0.85){\tiny{$\Omega_3$}};
			\node at (0.48,0.75){\tiny{$\Omega_5$}};
			\node at (0.52,0.25){\tiny{$\Omega_5$}};
			\node at (0.5,0.15){\tiny{$\Omega_3$}};
			\node at (0.72,0.48){\tiny{$\Omega_4$}};
			\node at (0.8,0.5){\tiny{$\Omega_6$}};
			\node at (0.28,0.52){\tiny{$\Omega_4$}};
			\node at (0.2,0.5){\tiny{$\Omega_6$}};
			\node at (0.68,0.362){\tiny{$S_1$}};
			\node at (0.68,0.638){\tiny{$S_2$}};
			\node at (0.32,0.362){\tiny{$S_4$}};
			\node at (0.32,0.638){\tiny{$S_3$}};
			\filldraw [black] (0.717,0.35) circle (1pt);
			\filldraw [black] (0.717,0.65) circle (1pt);
			\filldraw [black] (0.283,0.35) circle (1pt);
			\filldraw [black] (0.283,0.65) circle (1pt);
			\node at (0.66,0.43){\tiny{$\bar l_{12}$}};
			\node at (0.66,0.57){\tiny{$\bar l_{21}$}};
			\node at (0.9,0.51){\tiny{$\tilde L_{12}$}};
			\node at (0.79,0.66){\tiny{$l_{21}$}};
			\node at (0.79,0.34){\tiny{$l_{12}$}};
			\filldraw [black] (0.5,0.5) circle (1.3pt) node [above]{0};
		},
	},
}
\begin{axis}[
hide x axis,
hide y axis,
xmin=-2,
xmax=2,
]
\addplot[domain=0:2*pi, samples=200]
({sin(deg(x))},{cos(deg(x))});
\addplot[domain=1.25*pi:1.5*pi, ->, samples=200]
({sin(deg(x))},{cos(deg(x))});
\addplot[domain=0.25*pi:0.5*pi, ->, samples=200]
({sin(deg(x))},{cos(deg(x))});
\addplot[domain=1.75*pi:2*pi, -<, samples=200]
({sin(deg(x))},{cos(deg(x))});
\addplot[domain=0.75*pi:pi, -<, samples=200]
({sin(deg(x))},{cos(deg(x))});
\addplot[domain=0:sqrt(3)/2+1/2*cos(15),blue]
{-tan(15)*x+1/2+sqrt(3)/2*tan(15)};
\addplot[domain=0:1/6,-<,blue]
{-tan(15)*x+1/2+sqrt(3)/2*tan(15)};
\addplot[domain=0:sqrt(3)/2+1/2*cos(15),blue]
{tan(15)*x-1/2-sqrt(3)/2*tan(15)};
\addplot[domain=-sqrt(3)/2-1/2*cos(15):0,blue]
{tan(15)*x+1/2+sqrt(3)/2*tan(15)};
\addplot[domain=-sqrt(3)/2-1/2*cos(15):0, blue]
{-tan(15)*x-1/2-sqrt(3)/2*tan(15)};
\addplot[domain=-1/6:0,>-, blue]
{-tan(15)*x-1/2-sqrt(3)/2*tan(15)};
\addplot[domain=sqrt(3)/2-tan(15)/2:sqrt(3)/2+cos(75)/2,orange]
{x*tan(75)+1/2-tan(75)*sqrt(3)/2};
\addplot[domain=sqrt(3)/2-tan(15)/2+1/25:sqrt(3)/2+cos(75)/2,<-,orange]
{x*tan(75)+1/2-tan(75)*sqrt(3)/2};
\addplot[domain=sqrt(3)/2-tan(15)/2:sqrt(3)/2+cos(75)/2,orange]
{-x*tan(75)-1/2+tan(75)*sqrt(3)/2};
\addplot[domain=-sqrt(3)/2-cos(75)/2:-sqrt(3)/2+tan(15)/2,orange]
{x*tan(75)-1/2+tan(75)*sqrt(3)/2};
\addplot[domain=-sqrt(3)/2-cos(75)/2:-1/25-sqrt(3)/2+tan(15)/2,->,orange]
{x*tan(75)-1/2+tan(75)*sqrt(3)/2};
\addplot[domain=-sqrt(3)/2-cos(75)/2:-sqrt(3)/2+tan(15)/2,orange]
{-x*tan(75)+1/2-tan(75)*sqrt(3)/2};
\addplot[domain=-sqrt(3)/2-cos(75)/2:sqrt(3)/2+cos(75)/2,orange]
{sqrt(5/4+1/(sqrt(2))-x^2)};
\addplot[domain=0:sqrt(3)/2+cos(75)/2,<-,orange]
{sqrt(5/4+1/(sqrt(2))-x^2)};
\addplot[domain=-sqrt(3)/2-cos(75)/2:sqrt(3)/2+cos(75)/2,orange]
{-sqrt(5/4+1/(sqrt(2))-x^2)};
\addplot[domain=-sqrt(3)/2-cos(75)/2:0,->,orange]
{-sqrt(5/4+1/(sqrt(2))-x^2)};
\addplot[domain=-1/2+1/2*sin(15):1/2-1/2*sin(15),blue]
(sqrt(5/4+1/(sqrt(2))-x^2),x);
\addplot[domain=0:1/2-1/2*sin(15),<-,blue]
(sqrt(5/4+1/(sqrt(2))-x^2),x);
\addplot[domain=-1/2+1/2*sin(15):1/2-1/2*sin(15),blue]
(-sqrt(5/4+1/(sqrt(2))-x^2),x);
\addplot[domain=-1/2+1/2*sin(15):0,->,blue]
(-sqrt(5/4+1/(sqrt(2))-x^2),x);
\end{axis}
\end{tikzpicture}

Denoting
 $$d=2^{-1}\min(\sqrt{2+n/t},\sqrt{2-n/t}),$$
and making  the   decomposition
\begin{align}
\rho&=R+h_I+h_{II},\label{eq:36}\\
\bar\rho&=\bar R+\bar h_{I}+\bar h_{II},\label{eq:37s}
\end{align}
we then can  obtain the following estimates
\begin{align}
\left|e^{-2\varphi}h_I\right|&\leq Ct^{\frac{1}{2}-l} &\text{ on $\Sigma$},\label{eq:38s}\\
\left|e^{-2\varphi}h_{II}\right|&\leq Ct^{-q/2} &\text{ on $L$},\label{eq:39s}\\
\left|e^{-2\varphi} R(z)\right|&\leq Ce^{-Ct\epsilon^2} &\text{ on $L^\epsilon$},\label{eq:40s}\\
\left|e^{2\varphi}\bar h_I\right|&\leq Ct^{\frac{1}{2}-l} &\text{ on $\Sigma$},\label{eq:41s}\\
\left|e^{2\varphi}\bar h_{II}\right|&\leq Ct^{-\frac{q}{2}} &\text{ on $\bar L$},\label{eq:42s}\\
\left|e^{2\varphi}\bar R(z)\right|&\leq Ce^{-Ct\epsilon^2} &\text{ on $\bar L^\epsilon$},\label{eq:43s}
\end{align}
where we can see $L$, $\bar L$ in  {Figure 4} and for a positive real number $\epsilon$ ($<d/2$)
\begin{align*}
L^\epsilon &=\{z\in L:\left|z-S_j\right|\geq\epsilon \text{ for }j=1,2,3,4\},\\
\bar L^\epsilon &=\{z\in \bar L:\left|z-S_j\right|\geq\epsilon \text{ for }j=1,2,3,4\}.
\end{align*}

In the following, we first  get down to the decomposition of $\rho$ on the  one side of  $\wideparen{S_1S_2}$;  then  we shall study the decomposition of $\bar\rho$ on the other side of $\wideparen{S_1S_2}$; At last,
 in the similar way, we would investigate the decomposition on the other curves, like $\wideparen{S_2S_3}$, $\wideparen{S_3S_4}$ and $\wideparen{S_4S_1}$.

\subsection{Decomposition of $\rho$ on $\wideparen{S_1S_2}$}
\indent

We consider $\rho$ as a function of $\theta$. If we set $\theta_0=-arg(A)$, then the region of $\rho$ is $(-\theta_0,\theta_0)$, which is  a symmetric interval about $\theta=0$.  Thu.s, we could decompose $\rho$ as the sum of an odd function and an even function
\begin{equation}
 \rho(\theta)=H_e(\theta^2)+\theta H_o(\theta^2),\notag
\end{equation}
$H_o$ and $H_e$  can be   approximated  by Taylor's expansion at the point $\theta^2=\theta_0^2$
\begin{align}
 H_o(\theta^2) &=\mu_0^o+\cdots+\mu_k^o\left(\theta^2-\theta_0^2\right)^k+O(\theta^2-\theta_0^2)^{k+1}, \notag\\
 H_e(\theta^2) &=\mu_0^e+\cdots+\mu_k^e\left(\theta^2-\theta_0^2\right)^k+O(\theta^2-\theta_0^2)^{k+1}. \notag
\end{align}
Define
\begin{align}
 R(\theta)  &=\sum_{l=0}^k\left(\mu_l^e(\theta^2-\theta_0^2)^l+\theta \mu_l^o(\theta^2-\theta_0^2)^l\right) ,\notag \\
 h(\theta)  &=\rho(\theta)-R(\theta) ,\notag  \\
 \alpha(z) &=\left(z-S_1\right)^q\left(z-S_2\right)^q,\notag
\end{align}
where  $k$ and $q$ are two fixed positive integers and  have the relationship $$k=4q+1.$$

Notice that $$R(\pm\theta_0)=\rho(\pm\theta_0), \ \ \theta=-i\log z, \ {\rm on } \ \wideparen{S_1S_2}, $$
we   consider  $R$ as a function of complex number $z$, which could be extended analytically to a fairly large neighborhood of $\wideparen{S_1S_2}$.

We consider a function $$\psi=\varphi/(it),$$
then  (\ref{eq:27}) implies that     $\psi$ strictly  decrease  on  $(-\theta_0,\theta_0)$.
 Thus, we  can  consider $\theta$ as a function of $\psi$ and   $h/\alpha$, which are defined by
\begin{equation}
 (h/\alpha)(\psi)=
  \begin{cases}
   h\left(\theta(\psi)\right)/\alpha\left(\theta(\psi)\right), &\psi(\theta_0)\leq \psi \leq\psi(-\theta_0), \\
   0, &\text{otherwise on the real line}.
  \end{cases}\notag
\end{equation}
We verify that
$$(h/\alpha)(\theta)=O((\theta\pm\theta_0)^{k+1-q}), \ \ \theta\to\pm\theta_0.$$

Since both $S_1$ and $S_2$ are the first order stationary phase points,
   one   gets
   $$\frac{d\theta}{d\psi}=\left(\frac{d\varphi}{itd\theta}\right)^{-1}=O((\theta\pm\theta_0)^{-1}), \ \theta\to\pm\theta_0.$$
 For any integer $1\leq l\leq \frac{3q+2}{2} $, we deduce that
 \begin{equation}\label{eq:30s}
 h/\alpha\in H_\psi^l,
 \end{equation}
where  $H_\psi^l$ norm is uniformly bounded with respect to  $(n,t)$ as long as $|n|\leq V_0t$.

We apply Fourier transform to $(h/\alpha)(\theta)$ to decompose $h$ into two parts
 $$h=h_I+h_{II},$$
 where
\begin{align}
 h_I(\theta)&=\alpha(\theta)\int_t^\infty e^{is\psi(\theta)}(\widehat{h/\alpha})(s)ds,\label{eq:hI} \\
 h_{II}(\theta)&=\alpha(\theta)\int_{-\infty}^t e^{is\psi(\theta)}(\widehat{h/\alpha})(s)ds, \label{eq:hII}
\end{align}
and
 $$\widehat{\left(\frac{h}{\alpha}\right)}(s)=\int_{-\infty}^{\infty}e^{-is\psi}\left(\frac{h}{\alpha}\right)(s)ds.$$ Thus, we have $$\rho=R+h_I+h_{II}.$$
  By using (\ref{eq:30s})-(\ref{eq:hI}) and Schwartz inequality,
 we have
\begin{equation}\label{eq:25}
\left|e^{-2\varphi}h_I\right|\leq Ct^{\frac{1}{2}-l} \ \ \text{on $\wideparen{S_1S_2}$}.
\end{equation}

Let
\begin{align}
 p &=\sqrt{d^2+2d+\frac{1}{2}}-\sqrt{\frac{1}{2}} ,\notag
\end{align}
and introduce a   curve
\begin{equation}
 L_{12}=l_{12}\cup l_{21}\cup \tilde L_{12},
\end{equation}
where
\begin{align*}
 l_{12} &=\{S_1+S_1e^\frac{i\pi}{4}u: 0\leq u\leq p\}, \\
 l_{21} &=\{S_2+S_2e^{-\frac{i\pi}{4}}(p-u): 0\leq u\leq p\}, \\
 \tilde L_{12} &=arc(S_1+S_1e^\frac{i\pi}{4}p,S_2+S_2e^{-\frac{i\pi}{4}}p).
\end{align*}

From (\ref{eq:hII}),   It is shown that  $h_{II}$ can be  analytically  extend   to $\{Re\varphi>0\}$,
 and
\begin{equation}
\left|e^{-2\varphi}h_{II}\right|\leq Ce^{-tRe(i\psi))}\left|\alpha\right|.
\end{equation}
Since that both $S_1$ and $S_2$ are first order saddle points and $\tilde L_{12}$ is a closed curve in the inner of $\{Re\varphi>0\}$, we have
\begin{equation}\label{eq:28}
 Re{i\psi}\geq
 \begin{cases}
  Cu^2 &\text{on $l_{12}$}, \\
  C(p-u)^2 &\text{on $l_{21}$},\\
  C &\text{on $\tilde L_{12}$}.
 \end{cases}
\end{equation}
It is easy to check that
\begin{equation}\label{eq:37}
\left|\alpha\right|=
\begin{cases}
O(u^q) &\text{on $l_{12}$}, \\
O((p-u)^q) &\text{on $l_{21}$}, \\
O(1) &\text{on $\tilde L_{12}$}.
\end{cases}
\end{equation}
With (\ref{eq:28}) and (\ref{eq:37}), we obtain
\begin{equation}\label{eq:30}
 \left|e^{-2\varphi}h_{II}\right|\leq Ct^{-q/2}\  \text{ on $L_{12}$}.
\end{equation}
For sufficiently small positive number $\epsilon$ (\emph{$\epsilon<d/2$}), we define
\begin{align*}
 l_{12}^\epsilon &=\{z\in l_{12}:\left|z-S_1\right|\geq\epsilon\}, \\
 l_{21}^\epsilon &=\{z\in l_{21}:\left|z-S_2\right|\geq\epsilon\}, \\
 L_{12}^\epsilon &=l_{12}^\epsilon\cup  l_{21}^\epsilon\cup\tilde L_{12}.
\end{align*}
Since   $\text{dist}(L_{12}^\epsilon,\{S_1,S_2\})\geq\epsilon,$
it follows that
  $$Re(i\psi)\geq C\epsilon^2,$$
and with the boundedness of $R(z)$ on $L_{12}^\epsilon$, we get
\begin{equation}\label{eq:31}
 \left|e^{-2\varphi}R(z)\right|\leq Ce^{-Ct\epsilon^2}\ \  \text{on $L_{12}^\epsilon$}.
\end{equation}

\subsection{Decomposition of $\bar\rho$ on $\wideparen{S_1S_2}$}
\indent

For the one side of $\wideparen{S_1S_2}$, we have decomposed $\rho$ into three parts $$\rho=R+h_I+h_{II}.$$
Similar result  can be obtained   for $\bar\rho$ on  $\wideparen{S_1S_2}$ and set
$$
\bar p=\frac{\sqrt{2(2-\frac{n}{t})}}{\sqrt{2+\frac{n}{t}}+\sqrt{2-\frac{n}{t}}}.
$$
Consider  the contour   $\bar L_{12}=\bar l_{12}\cup\bar l_{21}$, where
\begin{align*}
 \bar l_{12}&= \{S_1-S_1e^{-\frac{i\pi}{4}}u: 0\leq u\leq \bar p\} ,\\
 \bar l_{21}&=:\{S_2-S_2e^{\frac{i\pi}{4}}(\bar p-u): 0\leq u\leq \bar p\}.
\end{align*}

  Noticing  decomposition    $$\bar\rho=\bar R+\bar h_{I}+\bar h_{II},$$ where
$$
 \bar R(z)=\overline{R(\Bar z^{-1})},
$$
 then in a similar  way  to the derivation of   $h_I$ and $h_{II}$,
 we can get the estimates
\begin{equation}\label{eq:32}
 \left|e^{2\varphi}\bar h_I\right|\leq Ct^{\frac{1}{2}-l}\ \ \text{ on $\wideparen{S_1S_2}$},
\end{equation}
\begin{equation}\label{eq:33}
 \left|e^{2\varphi}\bar h_{II}\right|\leq Ct^{-\frac{q}{2}}\text{ on $\bar L_{12}$}.
\end{equation}
Let
\begin{align*}
 \bar l_{12}^\epsilon &=\{z\in \bar l_{12}:\left|z-S_1\right|\geq\epsilon\}, \\
 \bar l_{21}^\epsilon &=\{z\in \bar l_{21}:\left|z-S_2\right|\geq\epsilon\}, \\
 \bar L_{12}^\epsilon &=l_{12}^\epsilon\cup  l_{21}^\epsilon,
\end{align*}
then for a  fixed $\epsilon$, we get
\begin{equation}\label{eq:34}
 \left|e^{2\varphi}\bar R(z)\right|\leq Ce^{-Ct\epsilon^2} \text{on $\bar L_{12}^\epsilon$}.
\end{equation}

\newtheorem{remark}[proposition]{Remark}
 \begin{remark}\label{re:4.1}
  We have decomposed $\rho$ and $\bar\rho$ on $\wideparen{S_1S_2}$; and we could get (\ref{eq:36})-(\ref{eq:37s}) on $\wideparen{S_2S_3}$, $\wideparen{S_3S_4}$ and $\wideparen{S_4S_1}$: $\rho=R+h_I+h_{II}$, similarly. What's more, we could get (\ref{eq:38s})-(\ref{eq:43s}) in the same way.
 \end{remark}

\subsection{The second RHP deformation }
\indent

According to the decomposition for $\rho$ and $\bar\rho$, we would make the second  deformation  to get a new  RHP equivalent to the original one.   Set
$$b_\pm=b_\pm^e+b_\pm^o,$$ where
\begin{align*}
b_+^e=&\left(
\begin{matrix}
1 & (R+h_{II})\delta^2e^{-2\varphi} \\
0 & 1
\end{matrix}\right),&
b_+^o=&\left(
\begin{matrix}
1 & h_I\delta^2e^{-2\varphi} \\
0 & 1
\end{matrix}\right), \\
b_-^e=&\left(
\begin{matrix}
1 & 0 \\
(\bar R+\bar h_{II})\delta^{-2}e^{2\varphi} & 1
\end{matrix}\right),&
b_-^o=&\left(
\begin{matrix}
1 & 0 \\
\bar h_I\delta^{-2}e^{2\varphi} & 1
\end{matrix}\right).
\end{align*}
Then, for (\ref{eq:34s})-(\ref{eq:37s}), we have $$v^{(1)}=(b_-^ob_-^e)^{-1}b_+^ob_+^e \text{ on } \Sigma^{(1)}.$$ Since both $h_{II}$ and $R$ is analytic on $\Omega_5$ and $\Omega_6$, $b_+^e$ is analytic on $\Omega_5$ and $\Omega_6$; similarly, $b_-^e$ is analytic on $\Omega_3$ and $\Omega_4$. Thus, we  define a new holomorphic function on $\mathbb{C}\setminus\Sigma^{(2)}$   as shown in  {Figure 4}:

\begin{equation}\label{eq:35}
 m^{(2)}=
 \begin{cases}
  m^{(1)} & \text{on $\Omega_1$ and $\Omega_2$}, \\
  m^{(1)}\left(b_-^e\right)^{-1} & \text{on $\Omega_3$ and $\Omega_4$},\\
  m^{(1)}\left(b_+^e\right)^{-1} & \text{on $\Omega_5$ and $\Omega_6$},
 \end{cases}
\end{equation}
 and it is  easy to see    that
$$\lim_{z\to\infty}m^{(2)}=\lim_{z\to\infty}m^{(1)}=I.$$

From (\ref{eq:19}) and (\ref{eq:35}), it follows that
\begin{equation}
 m_+^{(2)}=m_-^{(2)}v^{(2)}\  \text{ on $\Sigma^{(2)},$}
\end{equation}
where
\begin{equation}
  v^{(2)}=(b_-^{(2)})^{-1}b_+^{(2)}
\end{equation}
and
\begin{align*}
b_+^{(2)}&=
\begin{cases}
b_+^e & \text{ on }L,\\
I & \text{ on }\bar L,\\
b_+^o & \text{ on }\Sigma,
\end{cases}&
b_-^{(2)}&=
\begin{cases}
I & \text{ on }L,\\
b_-^e & \text{ on }\bar L,\\
b_-^o & \text{ on }\Sigma
\end{cases}\\
\end{align*}
with
\begin{align*}
w_\pm^{(2)}&=\pm(b_\pm^{(2)}-I),\ \ w^{(2)}=w_+^{(2)}+w_-^{(2)}.
\end{align*}
We  have completed the second RHP transformation.

\section{Reduction of the RHP}\label{Sec:tro}
\indent

In this section, we would like to reduce the previous RHP to a model  RHP.      We  firstly  apply  an integral operator introduced  in  \cite{beals1984scattering}
 to write $m^{(2)}$ in the integral form.
 Then, we reduce this integral to the sum of the decaying part and the integral related to the leading order.
\subsection{Partition of matrices}
\indent

Recall that $m^{(2)}$ can be written in the integral form
\begin{equation}\label{eq:38}
m^{(2)}(z)=I+\frac{1}{2\pi i}\int_{\Sigma^{(2)}}\frac{((I-C_{w^{(2)}})^{-1}I)(\tau)w^{(2)}(\tau)}{\tau-z}d\tau,
\end{equation}
where  $C_{w^{(2)}}=C_+(\cdot w_-^{(2)})+C_-(\cdot w_+^{(2)})$, and  $C_\pm$ are the Cauchy operators defined by
 $$C_\pm(f)(z)=\frac{1}{2\pi i}\int_{\Sigma^{(2)}}\frac{f(\tau)d\tau}{\tau-z_\pm},\ \ f\in L^2(\Sigma^{(2)}),\ \ z\in\Sigma^{(2)}.$$
By (\ref{eq:21s}), (\ref{eq:22}), (\ref{eq:26}) and  (\ref{eq:35}), we  then get
\begin{alignat}{2}
q_n &=\frac{d}{dz}\left((m^{(2)})_{12}\delta^{-1}\right)\Big|_{z=0}\notag =\delta(0)^{-1}\frac{d}{dz}(m^{(2)})_{12}\Big|_{z=0}\notag \\
&=\frac{\delta(0)^{-1}}{2\pi i}\int_{\Sigma^{(2)}}z^{-2}\left[((I-C_{w^{(2)}})^{-1}I)(z)w^{(2)}(z)\right]_{12}dz.\label{eq:39}
\end{alignat}

Now we would like to decompose $w_\pm^{(2)}$ into two parts
 $$w_\pm^{(2)}=w_\pm'+w_\pm^c,$$
  where
\begin{align*}
 w_\pm^c&=w_\pm^a+w_\pm^b,\\
w_\pm^a & :
\begin{cases}
 w_\pm^a&=w_\pm^{(2)} \text{ on } \{|z|=1\}, \\
 w_\pm^a &\text{is equal to the contribution to } w_\pm^{(2)} \text{ from terms of
 	type }\\ &h_{II} \text{ and } \bar h_{II},
 \end{cases} \\
 w_\pm^b &:
 \begin{cases}
 w_\pm^b&=w_\pm^{(2)}-w_\pm^a \text{ on } L^\epsilon\cup\bar L^\epsilon,\\
 w_\pm^b&=0 \text{ on } \Sigma^{(2)}\setminus(L^\epsilon\cup\bar L^\epsilon).
 \end{cases}
\end{align*}
So we  set a contour consisting of four crosses $$\Sigma'=\Sigma^{(2)}\setminus(\Sigma\cup L^\epsilon\cup\bar L^\epsilon)=\cup_{j=1}^4\Sigma_j,$$ where $\Sigma_j$ is the small cross connected to $S_j$ for $j=1, 2, 3, 4$. Note that the orientation of $\Sigma'$ and $\Sigma_j$
 are determined by that of $\Sigma^{(2)}$ (see  {Figure 5}).

 \ \ \ \ \ \ \ \ \ \ \ \ \ \ \ \ \ \
 \begin{tikzpicture}
 \pgfplotsset{
 	every axis/.append style={
 		extra description/.code={
 			\node at (0.5,0) { {Figure 5.} \emph{$\Sigma'$ consisting of four crosses}};
 			\node at (0.92,0.33){$\Sigma_1$};
 			\node at (0.92,0.67){$\Sigma_2$};
 			\node at (0.08,0.67){$\Sigma_3$};
 			\node at (0.08,0.33){$\Sigma_4$};
 			\filldraw [black] (0.803,0.298) circle (1.3pt);
 			\filldraw [black] (0.197,0.298) circle (1.3pt);
 			\filldraw [black] (0.803,0.702) circle (1.3pt);
 			\filldraw [black] (0.197,0.702) circle (1.3pt);
 			\filldraw [black] (0.5,0.5) circle (1.3pt) node [above] {0};
 		}
 	}
 }
 \begin{axis}[
 hide x axis,
 hide y axis,
 xmin=-1.43,
 xmax=1.43,
 ymin=-1.23,
 ymax=1.23,
 ]
 \addplot[domain=0:2*pi, samples=200,/tikz/dashed]
 ({sin(deg(x))},{cos(deg(x))});
 \addplot[domain=sqrt(3)/2-sin(75)/5:sqrt(3)/2+sin{75}/5]
 {-tan(15)*x+1/2+sqrt(3)/2*tan(15)};
 \addplot[domain=sqrt(3)/2-sin(75)/5:sqrt(3)/2-1/2*sin{75}/5,-<]
 {-tan(15)*x+1/2+sqrt(3)/2*tan(15)};
 \addplot[domain=sqrt(3)/2+sin(75)/10:sqrt(3)/2+sin{75}/5,>-]
 {-tan(15)*x+1/2+sqrt(3)/2*tan(15)};
 \addplot[domain=sqrt(3)/2-sin(75)/5:sqrt(3)/2+sin{75}/5]
 {tan(15)*x-1/2-sqrt(3)/2*tan(15)};
 \addplot[domain=sqrt(3)/2-sin(75)/5:sqrt(3)/2-sin{75}/10,->]
 {tan(15)*x-1/2-sqrt(3)/2*tan(15)};
 \addplot[domain=sqrt(3)/2+sin(75)/10:sqrt(3)/2+sin{75}/5,<-]
 {tan(15)*x-1/2-sqrt(3)/2*tan(15)};
 \addplot[domain=-sqrt(3)/2-sin(75)/5:-sqrt(3)/2+sin{75}/5]
 {tan(15)*x+1/2+sqrt(3)/2*tan(15)};
 \addplot[domain=-sqrt(3)/2-sin(75)/5:-sqrt(3)/2-sin{75}/10,->]
 {tan(15)*x+1/2+sqrt(3)/2*tan(15)};
 \addplot[domain=-sqrt(3)/2+sin(75)/10:-sqrt(3)/2+sin{75}/5,<-]
 {tan(15)*x+1/2+sqrt(3)/2*tan(15)};
 \addplot[domain=-sqrt(3)/2-sin(75)/5:-sqrt(3)/2+sin{75}/5]
 {-tan(15)*x-1/2-sqrt(3)/2*tan(15)};
 \addplot[domain=-sqrt(3)/2-sin(75)/5:-sqrt(3)/2-sin{75}/10,-<]
 {-tan(15)*x-1/2-sqrt(3)/2*tan(15)};
 \addplot[domain=-sqrt(3)/2+sin(75)/10:-sqrt(3)/2+sin{75}/5,>-]
 {-tan(15)*x-1/2-sqrt(3)/2*tan(15)};
 \addplot[domain=sqrt(3)/2-sin(15)/5:sqrt(3)/2+sin(15/5)]
 {x*tan(75)+1/2-tan(75)*sqrt(3)/2};
 \addplot[domain=sqrt(3)/2-sin(15)/5:sqrt(3)/2-sin(15)/10,-<]
 {x*tan(75)+1/2-tan(75)*sqrt(3)/2};
 \addplot[domain=sqrt(3)/2+sin(15)/10:sqrt(3)/2+sin(15)/5,>-]
 {x*tan(75)+1/2-tan(75)*sqrt(3)/2};
 \addplot[domain=sqrt(3)/2-sin(15)/5:sqrt(3)/2+sin(15/5)]
 {-x*tan(75)-1/2+tan(75)*sqrt(3)/2};
 \addplot[domain=sqrt(3)/2-sin(15)/5:sqrt(3)/2-sin(15)/10,->]
 {-x*tan(75)-1/2+tan(75)*sqrt(3)/2};
 \addplot[domain=sqrt(3)/2+sin(15)/10:sqrt(3)/2+sin(15)/5,<-]
 {-x*tan(75)-1/2+tan(75)*sqrt(3)/2};
 \addplot[domain=-sqrt(3)/2-sin(15)/5:-sqrt(3)/2+sin(15)/5]
 {x*tan(75)-1/2+tan(75)*sqrt(3)/2};
 \addplot[domain=-sqrt(3)/2-sin(15)/5:-sqrt(3)/2-sin(15)/10,-<]
 {x*tan(75)-1/2+tan(75)*sqrt(3)/2};
 \addplot[domain=-sqrt(3)/2+sin(15)/10:-sqrt(3)/2+sin(15)/5,>-]
{x*tan(75)-1/2+tan(75)*sqrt(3)/2};
 \addplot[domain=-sqrt(3)/2-sin(15)/5:-sqrt(3)/2+sin(15/5)]
 {-x*tan(75)+1/2-tan(75)*sqrt(3)/2};
 \addplot[domain=-sqrt(3)/2-sin(15)/5:-sqrt(3)/2-sin(15)/10,->]
 {-x*tan(75)+1/2-tan(75)*sqrt(3)/2};
 \addplot[domain=-sqrt(3)/2+sin(15)/10:-sqrt(3)/2+sin(15)/5,<-]
 {-x*tan(75)+1/2-tan(75)*sqrt(3)/2};
 \end{axis}
 \end{tikzpicture}

\newtheorem{lemma}[proposition]{Lemma}
\begin{lemma}
 There exists  a constant $C> 0$, such that
 \begin{align}
 \left|w_\pm^a\right|, \left|w^a\right|&\leq Ct^{-1}\  \text{ on $\Sigma^{(2)}$},\label{eq:40}\\
 \left|w_\pm^b\right|, \left|w^b\right|&\leq Ce^{-\gamma_\epsilon t}\  \text{ on $\Sigma^{(2)}$}.\label{eq:41}
\end{align}
Moreover, we have
\begin{align}
  \left\|w_\pm'\right\|_{L^2(\Sigma')}, \left\|w'\right\|_{L^2(\Sigma')}&\leq Ct^{-1/4}, \label{eq:42}\\
  \left\|w_\pm'\right\|_{L^1(\Sigma')}, \left\|w'\right\|_{L^1(\Sigma')}&\leq Ct^{-1/2}. \label{eq:43}
\end{align}
\end{lemma}
\begin{proof}
Since  $\delta$ and $\delta^{-1}$ is bounded, from the estimate (\ref{eq:38s})-(\ref{eq:43s}),  we get (\ref{eq:40}) and (\ref{eq:41}),   and from (\ref{eq:28}),   we  get

 $$|\delta^2Re^{-2\varphi}|\leq Const.e^{-Ct|z-S_1|^2}\  \text{ on } l_{12}\cap\Sigma',$$
which leads to
  \begin{align*}
 \int_{l_{12}\cap\Sigma'}|\delta^2Re^{-2\varphi}|dz&\leq Const.t^{-1/2},\\
 \int_{l_{12}\cap\Sigma'}|\delta^2Re^{-2\varphi}|^2dz&\leq Const.t^{-1/2}.
 \end{align*}
Since   all $S_j$ are the first order stationary phase point of $\varphi$,   we   obtain that
  \begin{align*}
 \int_{L\cap\Sigma'}|\delta^2Re^{-2\varphi}|dz&\leq Const.t^{-1/2},\\
 \int_{L\cap\Sigma'}|\delta^2Re^{-2\varphi}|^2dz&\leq Const.t^{-1/2}.
 \end{align*}
   For $\delta^{-2}\bar Re^{2\varphi}$ on $\bar L\cap\Sigma'$,  we   also have
  \begin{align*}
  \int_{\bar L\cap\Sigma'}|\delta^{-2}\bar Re^{2\varphi}|dz&\leq Const.t^{-1/2},\\
  \int_{\bar L\cap\Sigma'}|\delta^{-2}\bar Re^{2\varphi}|^2dz&\leq Const.t^{-1/2}.
  \end{align*}
 Therefore, by simply analyzing the entry of $w_\pm'$, we deduce (\ref{eq:42}) and (\ref{eq:43}).
\end{proof}

\subsection{Some resolvent and estimates}
\indent

We  shall    decompose  $$\int_{\Sigma^{(2)}}z^{-2}((I-C_{w^{(2)}})^{-1}I)(z)w^{(2)}(z)dz,$$
into the principal part and the decaying part
by the resolvent identity,

Under the assumption that both $(I-C_{w^{(2)}})^{-1}$ and $(I-C_{w'})^{-1}$ exist and are bounded, by the second resolvent identity, we have
\begin{align}\label{eq:44}
 &\int_{\Sigma^{(2)}}z^{-2}(I-C_{w^{(2)}})^{-1}Iw^{(2)}\notag\\
 =&\int z^{-2}(I-C_{w'})^{-1}Iw'+\int z^{-2}w^c+\int z^{-2}(I-C_{w'})^{-1}(C_{w^c}I)w^{(2)}\notag\\
 &+\int z^{-2}(I-C_{w'})^{-1}(C_{w'}I)w^c+\int z^{-2}(I-C_{w'})^{-1}C_{w^c}(I-C_{w^{(2)}})^{-1}(C_{w^{(2)}}I)w^{(2)}\notag\\
 =&\int z^{-2}(I-C_{w'})^{-1}Iw'+I+II+III+IV.
\end{align}

Since the length of $\Sigma^{(2)}$ is finite, by (\ref{eq:40}) and (\ref{eq:41}), we have
\begin{align}
\parallel{w_\pm^c}\parallel_{L^s(\Sigma^{(2)})}, \parallel{w^c}\parallel_{L^s(\Sigma^{(2)})}&\leq Ct^{-1} \ \ \text{ for } s=1, 2.\label{eq:55s}
\end{align}
For   $0\notin\Sigma^{(2)}$,   $z^{-2}$ is bounded on $\Sigma^{(2)}$,  so  we obtain
$$|I|\leq Ct^{-1}.$$
Because the Cauchy integral operator is bounded on $L^2(\Sigma^{(2)})$ and
$$\parallel C_\pm\parallel_{L^2\to L^2}\leq 1,$$
we can  get that
\begin{align}\label{eq:56s}
\parallel C_{w^c}I\parallel_{L^2(\Sigma^{(2)})}&\leq\parallel C_+(w_-^c)\parallel_{L^2(\Sigma^{(2)})}+\parallel C_-(w_+^c)\parallel_{L^2(\Sigma^{(2)})}\notag\\
	&\leq\parallel w_-^c\parallel_{L^2(\Sigma^{(2)})}+\parallel w_+^c\parallel_{L^2(\Sigma^{(2)})} \leq Ct^{-1}.
\end{align}

Similarly, from (\ref{eq:42}),
we have
\begin{equation}
\parallel C_{w'}I\parallel_{L^2(\Sigma^{(2)})}=O(t^{-1/4})\ \ \text{ as } t\to \infty. \label{eq:57}
\end{equation}
From   (\ref{eq:40})-(\ref{eq:41}) and $\parallel C_\pm\parallel_{L^2\to L^2}\leq 1$,   for $f\in L^2(\Sigma^{(2)})$ , we can show  that
\begin{align}
\parallel C_{w^c}f\parallel_{L^2(\Sigma^{(2)})}&\leq \parallel C_+(fw_-^c)\parallel_{L^2(\Sigma^{(2)})}+\parallel C_-(fw_+^c)\parallel_{L^2(\Sigma^{(2)})}\nonumber\\
&\leq C\parallel f\parallel_{L^2(\Sigma^{(2)})}(\sup_{z\in\Sigma^{(2)}}|w_-^c(z)|+\sup_{z\in\Sigma^{(2)}}|w_+^c(z)|)\nonumber\\
&\leq Ct^{-1}\parallel f\parallel_{L^2(\Sigma^{(2)})}, \label{eq:58}
\end{align}
by which,  we obtain that
\begin{align*}
|II&|\leq C\parallel C_{w^c}I\parallel_{L^2}\parallel w^{(2)}\parallel_{L^2}\\
&\leq C\parallel C_{w^c}I\parallel_{L^2}(\parallel w'\parallel_{L^2}+\parallel w^c\parallel_{L^2})\\
&\leq Ct^{-1},\\
|III|&\leq C\parallel C_{w'}I\parallel_{L^2}\parallel w^c\parallel_{L^2}\\
&\leq Ct^{-1},\\
|IV|&\leq C\parallel C_{w^c}\parallel_{L^2\to L^2}\parallel C_{w^{(2)}}I\parallel_{L^2}\parallel w^{(2)}\parallel_{L^2}\\
&\leq C\parallel C_{w^c}\parallel_{L^2\to L^2}(\parallel C_{w^c}I\parallel_{L^2}+\parallel C_{w'}I\parallel_{L^2})(\parallel w'\parallel_{L^2}+\parallel w^c\parallel_{L^2})\\
&\leq Ct^{-1}.
\end{align*}

From these estimates, combining (\ref{eq:39}) and  (\ref{eq:44}),  we  obtain  the following equation
\begin{equation}\label{eq:46}
  q_n=\frac{\delta(0)^{-1}}{2\pi i}\int_{\Sigma^{(2)}}z^{-2}\left[(I-C_{w'})^{-1}Iw'\right]_{12}(z)dz+O(t^{-1}).
\end{equation}
Let
 $$C_{w'}^{\Sigma'}=C_+(\cdot\times w_-'\big|_{\Sigma'})+C_-(\cdot\times w_-'\big|_{\Sigma'}),$$
 which is an operator on $2\times2$ Hilbert space $L^2(\Sigma')$.
 Since $w'$ vanishes on $\Sigma^{(2)}\setminus\Sigma'$,  by (2.58) in \cite{deift1993steepest}, we could write (\ref{eq:46}) as
\begin{equation}\label{eq:47}
  q_n=\frac{\delta(0)^{-1}}{2\pi i}\int_{\Sigma'}z^{-2}\left[(I-C_{w'}^{\Sigma'})^{-1}Iw'\right]_{12}(z)dz+O(t^{-1}),
\end{equation}
where  $C_{w'}$ in (\ref{eq:46}) is an operator on $L^2(\Sigma^{(2)})$, whose kernel is $w'$.

\begin{remark}
 We have assumed in this subsection that both $(I-C_{w^{(2)}})^{-1}$ and $(I-C_{w'})^{-1}$ exist and are bounded. In fact, in {Section 9}, we will prove that both $(I-C_{w^{(2)}})^{-1}$ and $(I-C_{w'})^{-1}$ exist and are bounded uniformly as $t\to\infty$.
\end{remark}

\section{Four crosses}\label{Sec:fou}
\indent

 In this section, we may decompose $w'$ into four parts according to the cross. Moreover, we will
 make some estimates  such  that the principal part  is  more accurate for $q_n$.

Define matrix functions on $\Sigma'$
\begin{equation}\label{eq:48}
w^j=
\begin{cases}
w'&\text{on $\Sigma_j$},\\
0 &\text{on $\Sigma'\setminus \Sigma_j$},
\end{cases}
\ \   j=1,2,3,4,
\end{equation}
then  we define integral operators on $L^2(\Sigma')$  with kernels in (\ref{eq:48})
\begin{equation}\label{eq:49}
A_j=C_{w^j}^{\Sigma'}, j=1,2,3,4.
\end{equation}

For $A_j$, we have the following result.
\begin{proposition}\label{pr:6.1}
 Given $j,k=1,2,3,4$ ($j\ne k$), we have
 \begin{eqnarray}
  \|A_jA_k\|_{L^2\to L^2}\leq Ct^{-\frac{1}{2}}, \label{eq:51}\\
  \|A_jA_k\|_{L^\infty\to L^2}\leq Ct^{-\frac{3}{4}}. \label{eq1:52}
 \end{eqnarray}
\end{proposition}
The proof of this proposition is similar to Lemma 3.5  in \cite{deift1993steepest}. Referring to (\ref{eq:42}) and (\ref{eq:43}), we only have to replace $A'$,  $B'$ in Lemma 3.5 of \cite{deift1993steepest} with $A_j$,  $A_k$ respectively; then, Proposition \ref{pr:6.1} follows.

Assuming that as $t\to\infty$, $(1-A_j)^{-1}$ exists and is uniformly bounded which  would be proven in   {Section \ref{S:9}}.
because  $C_{w'}^{\Sigma'}=\sum_jA_j$,   by direct calculation, we get
$$1-\sum_{j\neq k}A_jA_k(1-A_k)^{-1}=(1-C_{w'}^{\Sigma'})[1+\sum_jA_j(1-A_j)^{-1}].$$
We verify that for any $f\in L^(\Sigma')$,
\begin{align*}
\parallel A_j(f)\parallel_{L^2}&\leq \parallel C_+(fw_-^j)\parallel_{L^2}+\parallel C_-(fw_+^j)\parallel_{L^2}\\
&\leq (\parallel w_+^j\parallel_{L^\infty}+\parallel w_-^j\parallel_{L^\infty})\parallel f\parallel_{L^2}\\
&\leq2\parallel w_\pm'\parallel_{L^\infty}\parallel f\parallel_{L^2}\leq Ct^{-1}\parallel f\parallel_{L^2},
\end{align*}
which implies
\begin{equation}
\parallel A_j\parallel_{L^2\to L^2}=O(t^{-1}).\label{eq:65s}
\end{equation}
With {Proposition \ref{pr:6.1}}, as $t\to\infty$, we have
\begin{equation}\label{eq:51s}
 (1-C_{w'}^{\Sigma'})^{-1}=[1+\sum_jA_j(1-A_j)^{-1}][1-\sum_{j\neq k}A_jA_k(1-A_k)^{-1}]^{-1},
\end{equation}
then we  obtain  that
\begin{align}\label{eq:53}
 &\int z^{-2}[(1-C_{w'}^{\Sigma'})^{-1}I](z)w'(z)\notag\\
 =&\int z^{-2}[I+\sum_jA_j(1-A_j)^{-1}I](z)w'(z)dz+\int z^{-2}\{[1+\sum_jA_j(1-A_j)^{-1}]\notag\\
 &\times[1-\sum_{j\neq k}A_jA_k(1-A_k)^{-1}]^{-1}\sum_{j\neq k}A_jA_k(1-A_k)^{-1}I\}(z)w'(z)dz.
\end{align}
and
\begin{equation}\label{eq:53s}
A_jA_k(1-A_k)^{-1}=A_jA_k+A_jA_k(1-A_k)^{-1}A_k.
\end{equation}
 By the definition of $A_k$, we learn that $A_kI\in L^2(\Sigma')$ and
 $$\left\|A_kI\right\|_{L^2}\leq C\left\|w^k\right\|_{L^2}\leq C\left\|w'\right\|_{L^2}\leq Ct^{-1/4}.$$
 From the uniform boundedness of the operator $(1-A_k)^{-1}$, we obtain  $(1-A_k)^{-1}A_kI$ belong to $L^2(\Sigma')$ and
  \begin{equation}\label{eq:54s}
  	\left\|(1-A_k)^{-1}A_kI\right\|_{L^2}\leq Ct^{-1/4}.
  	\end{equation}
  Combining (\ref{eq:51}), (\ref{eq1:52}), (\ref{eq:53s}) and (\ref{eq:54s}), we get that
   $$[\sum_{j\neq k}A_jA_k(1-A_k)^{-1}]I\in L^2(\Sigma'),$$
    and
    $$\left\|[\sum_{j\neq k}A_jA_k(1-A_k)^{-1}]I\right\|_{L^2(\Sigma')}\leq Ct^{-3/4}.$$
    If we  can  verified that
\begin{equation}\label{eq:52}
\left\|[1+\sum_jA_j(1-A_j)^{-1}][1-\sum_{j\neq k}A_jA_k(1-A_k)^{-1}]^{-1}\right\|_{L^2\to L^2}
\end{equation}
is bounded uniformly as $t\to\infty$,  by  H\"older's inequality, we can write (\ref{eq:53})  in the form
\begin{align}\label{eq:54}
 &\int z^{-2}[(1-C_{w'}^{\Sigma'})^{-1}I](z)w'(z)\notag\\
 =&\int z^{-2}[I+\sum_jA_j(1-A_j)^{-1}I](z)w'(z)dz+O(t^{-1}).
\end{align}
With equation (\ref{eq:42})-(\ref{eq:43}),  we obtain that
 $$[1-\sum_{j\neq k}A_jA_k(1-A_k)^{-1}]^{-1}, \ \ t\to \infty$$ exists and is a bounded   operator on $L^2(\Sigma')$. Also, by (\ref{eq:65s}), we know  that  $$1+\sum_jA_j(1-A_j)^{-1}$$ exists and is bounded on $L^2(\Sigma')$. Thus, we have proven the uniform boundedness of  (\ref{eq:52}).

With (\ref{eq:53}),   if we  prove that for any pair of different numbers in $\{1,2,3,4\}$($j\neq k$) and some constant C
\begin{equation}\label{eq:55}
  \left|\int_{\Sigma'}z^{-2}[A_j(1-A_j)^{-1}I](z)w^k(z)\right|\leq Ct^{-1},
\end{equation}
and  have the following theorem.
\newtheorem{theorem}[proposition]{Theorem}
\begin{theorem}\label{thm:7.3}
The potential function  $q_n$  admits  the  asymptotic estimate
 \begin{equation}\label{eq:56}
  q_n=\frac{\delta(0)^{-1}}{2\pi i}\sum_j\int_{\Sigma_j}z^{-2}\{[(1-A_j)^{-1}I](z)w^j(z)\}_{12}dz+O(t^{-1}), \  t\to\infty.
 \end{equation}
\end{theorem}
\begin{proof}

By using  (\ref{eq:47}), (\ref{eq:54}) and (\ref{eq:55}), we could complete the proof of (\ref{eq:56}) by consider the following relation
 $$A_j(1-A_j)^{-1}=(1-A_j)^{-1}-1.$$
\end{proof}
To prove (\ref{eq:55}), considering
\[
 [A_j(1-A_j)^{-1}I]w^k=A_j(1-A_j)^{-1}A_jIw^k+A_jIw^k,
\]
we obtain  the  estimate
\begin{align*}
 \int_{\Sigma_k}|A_jIw^k|=&\int_{\Sigma_k}\left|\left(\int_{\Sigma_j}\frac{w^j(\eta)d\eta}{\eta-\zeta}\right)w^k(\zeta)\right|d\zeta \\
 \leq &C\parallel w^j\parallel_{L^1(\Sigma_j)}\parallel w^k\parallel_{L^1(\Sigma_k)} \leq Ct^{-1},
\end{align*}
\begin{align*}
  \int_{\Sigma_k}|A_j(1-A_j)^{-1}A_jIw^k|=&\int_{\Sigma_k}\left|\left(\int_{\Sigma_j}\frac{(1-A_j)^{-1}A_jI(\eta)w^j(\eta)d\eta}{\eta-\zeta}\right)w^k(\zeta)\right|d\zeta \\
  \leq&C\int_{\Sigma_j}|(1-A_j)^{-1}A_jI(\eta)w^j(\eta)|d\eta\parallel w^k\parallel_{L^1(\Sigma_k)} \\
  \leq&C\parallel(1-A_j)^{-1}A_jI\parallel_{L^2(\Sigma_j)}\parallel w^j\parallel_{L^1(\Sigma_j)}\parallel w^k\parallel_{L^1(\Sigma_k)} \\
  \leq&C\parallel A_jI\parallel_{L^2(\Sigma_j)}\parallel w^j\parallel_{L^2(\Sigma_j)}\parallel w^k\parallel_{L^1(\Sigma_k)} \\
  \leq &C\parallel A_jI\parallel_{L^2(\Sigma_j)}\parallel w'\parallel_{L^2(\Sigma_j)}^2\leq Ct^{-1}.
\end{align*}
Then (\ref{eq:55}) follows.

With Theorem \ref{thm:7.3},    the original RHP can be reduced to four RHPs on four separated crosses and the leading order is only about these four RHPs.

\section{Rotation and Scaling}
\indent

In this section, we would introduce scaling operators concerning each stationary phase point.

Before getting down to the scaling operator, we shall first make some preparations. From (\ref{eq:27}), we show that for any $j=1,2,3,4$,
\begin{equation}
\varphi''(S_j)=(-1)^j2iS_j^{-2}\sqrt{4t^2-n^2}.
\end{equation}
Let   $T_1=T_2=1$ and $T_3=T_4=-1$,   we    define the following quantities for $j=1,2,3,4$
\begin{align}
 \delta_j(z)&=e^{(-1)^{j-1}\frac{1}{2\pi i}\int_{T_j}^{S_j}\log(1-\left|r(\tau)\right|^2)\frac{d\tau}{\tau-z}},\\
  \nu_j&=-\frac{1}{2\pi}\log(1-\left|r(S_j)\right|^2),\ \ \
  l_j(z)=\int_{T_j}^{S_j}\frac{d\tau}{\tau-z},\\
  \chi_j(z)&=\frac{1}{2\pi i}\int_{T_j}^{S_j}\log\frac{1-\left|r(\tau)\right|^2}{1-\left|r(S_j)\right|^2}\frac{d\tau}{\tau-z},
\end{align}
where the symbol $\int_{T_j}^{S_j}$ stands for the integral on the arc $\wideparen{T_jS_j}$ from $T_j$ to $S_j$.
These quantities have the relationship
\begin{align}\label{eq:62}
\delta_j(z)&=e^{(-1)^{j-1}(\chi_j(z)+i\nu_jl_j(z))} \notag\\
&=\left(\frac{S_j-z}{T_j-z}\right)^{(-1)^{j-1}\nu_j}e^{(-1)^{j-1}\chi_j(z)}.
\end{align}
We verify that $\delta(z)=\prod_{j=1}^4\delta_j(z)$ by making product directly and define
\[
 \hat\delta_j(z)=\delta(z)/\delta_j(z).
\]

Since each $S_j$ is first order stationary phase point, we have
\begin{equation}\label{eq:80}
 \varphi(z)=\varphi(S_j)+\varphi''(S_j)(z-S_j)^2+\varphi_j(z),
\end{equation}
where $$\varphi_j(z)=O(\left|z-S_j\right|^3).$$
We extend the four small crosses to four infinite ones
\begin{align*}
\begin{cases}
\Sigma(S_j)=&(S_j+S_je^{i\pi/4}\mathbb{R})\cup(S_j+S_je^{-i\pi/4}\mathbb{R}),\\ &\text{ oriented inward like }\Sigma_j\text{ for }j=1,3,\\
\Sigma(S_j)=&(S_j+S_je^{i\pi/4}\mathbb{R})\cup(S_j+S_je^{-i\pi/4}\mathbb{R}),\\  &\text{ oriented outward like }\Sigma_j\text{ for }j=2,4,
\end{cases}
\end{align*}
and define the contours
\begin{align*}
\Sigma_j(0)=
\begin{cases}
(e^{i\pi/4}\mathbb{R})\cup(e^{-i\pi/4}\mathbb{R}), \text{ oriented inward like }\Sigma_j, &j=1,3,\\
(e^{i\pi/4}\mathbb{R})\cup(e^{-i\pi/4}\mathbb{R}), \text{ oriented outward like }\Sigma_j, &j=2,4.
\end{cases}
\end{align*}
In fact, we have a mapping from $\Sigma_j(0)$ to $\Sigma(S_j)$
\begin{align*}
M_j: \Sigma_j(0)&\to\Sigma(S_j),\\
z&\mapsto\beta_jz+S_j,
\end{align*}
where
\begin{align}
&\beta_j=\frac{1}{2}(4t^2-n^2)^{-1/4}iS_j(-1)^j. \label{pop5}
\end{align}
Directly calculating we find that $$\varphi''(S_j)\beta_j^2=(-1)^{j+1}\frac{i}{2}.$$

We introduce the scaling operator
\begin{align*}
 N_j: (C^0\cup L^2)(\Sigma(S_j))&\to (C^0\cup L^2)(\Sigma_j(0),) \\
 f(z)&\mapsto (N_jf)(z)=f((\beta_jz+S_j),
\end{align*}
which is the pull-back of $M_j$.

\ \ \ \ \ \ \ \ \ \ \ \ \ \ \ \ \
\begin{tikzpicture}
\draw[->] (3,2.8) -- (3.8,2.8);
\pgfplotsset{
	every axis/.append style={
		extra description/.code={
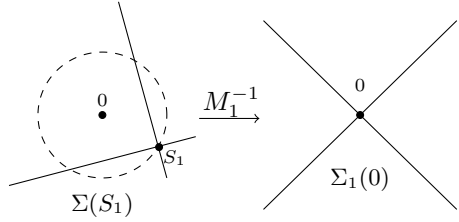
			\node at (0.5,0) { {Figure 6.} \emph{the mapping $M_1^{-1}$}};
			\filldraw [black] (0.25,0.5) circle (1.3pt) node [above] {\scriptsize{0}};
			\node at(0.39,0.4){\scriptsize{$S_1$}};
			\filldraw [black] (0.36,0.425) circle (1.3pt);
			\node at(0.75,0.57){\scriptsize{0}};
			\filldraw [black] (0.75,0.5) circle (1.3pt);
			\node at(0.5,0.53){$M_1^{-1}$};
			\node at(0.25,0.29){\small{$\Sigma(S_1)$}};
			\node at(0.75,0.35){\small{$\Sigma_1(0)$}};
		}
	}
}
\begin{axis}[
hide x axis,
hide y axis,
xmin=-4,
xmax=4,
ymin=-3.4,
ymax=3.4,
]
\addplot[domain=-2-1.5*cos(15):-2+1.5*cos(15)]
{tan(15)*(x+2-sqrt(3)/2)-1/2};
\addplot[domain=-1.75:-1]
{-tan(75)*(x+2-sqrt(3)/2)-1/2};
\addplot[domain=0:2*pi, samples=200,/tikz/dashed]
({cos(deg(x))-2},{sin(deg(x))});
\addplot[domain=0.5:3.5]
{x-2};
\addplot[domain=0.5:3.5]
{-x+2};
\end{axis}
\end{tikzpicture}

By equation  (\ref{eq:62}) and  (\ref{eq:80}), we obtain
\begin{align}\label{eq:64}
 N_j(e^{-\varphi})(z)&=e^{-\varphi(S_j)-\frac{\varphi''(S_j)}{2}\beta_j^2z^2-N_j\varphi_j(z)}\notag\\
 &=S_j^ne^{-\frac{t}{2}(S_j^2-S_j^{-2})+(-1)^j\frac{i}{4}z^2-N_j\varphi_j(z)},
 \end{align}
\begin{align}\label{eq:65}
 N_j\delta_j(z)=&\delta_j(\beta_jz+S_j) +\left(\frac{\beta_j}{\beta_jz+S_j-T_j}\right)^{(-1)^{j-1}i\nu_j}z^{(-1)^{j-1}i\nu_j}\notag\\&\times e^{(-1)^{j-1}\chi_j(\beta_jz+S_j)}.
\end{align}
Further from  (\ref{eq:64}) and (\ref{eq:65}),   we  get
\begin{equation}\label{eq:83s}
 N_j(\delta_je^{-\varphi})(z)=\delta_j^0\delta_j^1(z),
\end{equation}
where
\begin{align}
 \delta_j^0=&S_j^n\left(\frac{\beta_j}{S_j-T_j}\right)^{(-1)^{j-1}i\nu_j}e^{(-1)^{j-1}\chi_j(S_j)-\frac{t}{2}(S_j^2-S_j^{-2})}\hat\delta_j(S_j)\label{eq:96},\\
 \delta_j^1(z)=&\left(\frac{S_j-T_j}{\beta_jz+S_j-T_j}\right)^{(-1)^{j-1}i\nu_j}z^{(-1)^{j-1}i\nu_j}\frac{N_j\hat\delta_j(z)}{\hat\delta_j(S_j)}\notag\\
 &\times e^{(-1)^j\frac{i}{4}z^2-N_j\varphi_j(z)+(-1)^{j-1}(\chi_j(\beta_jz+S_j)-\chi_j(S_j))}.\label{eq:85}
\end{align}

 Note that $\Sigma(S_j)$ is extension of $\Sigma_j$. Let   $\hat w_\pm^j$ be the zero extension of $w_\pm^j\big|_{\Sigma_j}$ on $\Sigma(S_j)$;
then, the related operator on $L^2(\Sigma(S_j))$ is denoted as  $\hat A_j=C_{\hat w^j}$  with kernel
$$\hat w^j=\hat w_+^j+\hat w_-^j,\ \  j=1,2,3,4.$$

Define a $2\times2$ matrix $\Delta_j^0=(\delta_j^0)^{\sigma_3}$  and the related operator $\tilde\Delta_j^0$:   $\tilde\Delta_j^0\phi=\phi\Delta_j^0$ for $2\times2$ matrix $\phi$;
then $\tilde\Delta_j^0$ and its inverse are bounded.  Letting
\[
\tilde w_\pm^j=(\Delta_j^0)^{-1}(N_j\hat w_\pm^j)\Delta_j^0, \ \ \ \tilde w^j=\tilde w_+^j+\tilde w_-^j,
\]
and $\alpha_j=C_{\tilde w^j}: L^2(\Sigma_j(0))\to L^2(\Sigma_j(0))$,
direct  calculation  shows  that
\begin{align}
\alpha_j&=\tilde \Delta_j^0N_j\hat A_jN_j^{-1}(\tilde\Delta_j^0)^{-1},\label{eq:70}\\
\hat A_j&=N_j^{-1}(\tilde\Delta_j^0)^{-1}\alpha_j\tilde\Delta_j^0N_j.\label{eq:71}
\end{align}

Noticing that the support of $N_j\hat w^j$  belongs to  $M_j^{-1}\Sigma_j$,   by using  (\ref{eq:83}),  we obtain
\begin{equation}\label{eq:89}
(\Delta_j^0)^{-1}(N_j\hat w^j)\Delta_j^0=(\Delta_j^0)^{-1}(N_j\hat w_+^j)\Delta_j^0=\left(
\begin{matrix}
0 & R(\beta_jz+S_j)\delta_j^1(z)^2 \\
0 & 0
\end{matrix}\right).
\end{equation}
Similarly,   on $M_j^{-1}(\Sigma_j\cap \bar L)\setminus \{0\}$,  we  have
\begin{equation}\label{eq:90}
(\Delta_j^0)^{-1}(N_j\hat w^j)\Delta_j^0=(\Delta_j^0)^{-1}(N_j\hat w_-^j)\Delta_j^0=\left(
\begin{matrix}
0 & 0 \\
- \bar R(\beta_jz+S_j)\delta_j^1(z)^{-2} & 0
\end{matrix}\right).
\end{equation}

\section{Convergence}\label{S:8}\indent

Noticing  that  the principal part  of $q_n$ consists of four integrals respectively on four separate crosses: $\Sigma_j$ ($j=1,2,3,4$),  we  consider  the convergence of the principal part  in this section.
That is,  we estimate  four  terms
 $$\int_{\Sigma_j}z^{-2}[(1-A_j)^{-1}I](z)w^j(z)dz, \ j=1,2,3,4.\notag\\  $$
By (\ref{eq:71}), we have the following equalities
\begin{align}
 &\int_{\Sigma_j}z^{-2}[(1-A_j)^{-1}I](z)w^j(z)dz\notag\\
 =&\beta_j\int_{\Sigma_j(0)}(\beta_jz+S_j)^{-2}[(1-\alpha_j)^{-1}(\Delta_j^0)](z)(\Delta_j^0)^{-1}N_j\hat w^j(z)dz\notag\\
 =&\beta_j(\delta_j^0)^2\int_{\Sigma_j(0)}[(1-\alpha_j)^{-1}I](z)(\Delta_j^0)^{-1}N_j(\cdot\times\hat w^j)(z)\Delta_j^0dz.\label{eq:103}
\end{align}
We  estimate  the limit  of   formula (\ref{eq:103}).  For this purpose,
 we   investigate the convergence of  the following functions
\begin{align*}
&(\beta_jz+S_j)^{-2}N_jR(z)\delta_j^1(z)^2, & N_jR(z)\delta_j^1(z)^2,\\
&(\beta_jz+S_j)^{-2}N_j\bar R(z)\delta_j^1(z)^{-2}, & N_j\bar R(z)\delta_j^1(z)^{-2}.
\end{align*}
\begin{proposition}\label{Pro:1}
 For any arbitrary fixed constant $\gamma$ ($0< 2\gamma<1$), on $M_j^{-1}(\Sigma_j\cap L)\cap\{z:\pm z/e^{i\pi/4}> 0\}$ respectively, we have
  \begin{eqnarray}
   \left|(\beta_jz+S_j)^{-2}N_jR(z)\delta_j^1(z)^2-S_j^{-2}R(S_j\pm)e^{-iz^2/2}z^{2i\nu_j}\right| \notag\\
   \leq Ce^{-\frac{i}{2}\gamma z^2}t^{-\frac{1}{2}}\log t, \label{eq:104}\\
   \left|N_jR(z)\delta_j^1(z)^2-R(S_j\pm)e^{-iz^2/2}z^{2i\nu_j}\right| \notag\\
   \leq Ce^{-\frac{i}{2}\gamma z^2}t^{-\frac{1}{2}}\log t,\label{eq:105}
  \end{eqnarray}
where  $R(S_j+)=\bar r(S_j)$ on $M_j^{-1}(\Sigma_j\cap L)\cap\{z:  z/e^{i\pi/4}> 0\}$ and $R(S_j-)=-\frac{\bar r(S_j)}{1-\left|r(S_j)\right|^2}$ on $M_j^{-1}(\Sigma_j\cap L)\cap\{z:  z/e^{i\pi/4}< 0\}$.
\end{proposition}

This  proposition  can be proved in a similar way  to  proposition 10.1 in \cite{yamane2014long},
we omit it here.  Similarly, we also have the following proposition.
\begin{proposition}\label{Pro:2}
  For any arbitrary fixed constant $\gamma$ ($0< 2\gamma<1$), on $M_j^{-1}(\Sigma_j\cap\bar L)\cap\{z:\pm ze^{-i\pi/4}> 0\}$ respectively, we have
  \begin{eqnarray}
   \left|(\beta_jz+S_j)^{-2}N_j\bar R(z)\delta_j^1(z)^{-2}-S_j^{-2}\bar R(S_j\pm)e^{iz^2/2}z^{-2i\nu_j}\right| \notag\\
   \leq Ce^{\frac{i}{2}\gamma z^2}t^{-\frac{1}{2}}\log t, \label{eq:106}\\
   \left|N_j\bar R(z)\delta_j^1(z)^{-2}-\bar R(S_j\pm)e^{iz^2/2}z^{-2i\nu_j}\right|
   \leq Ce^{\frac{i}{2}\gamma z^2}t^{-\frac{1}{2}}\log t, \ j=1, 3,\label{eq:107}
  \end{eqnarray}
where $\bar R(S_j+)=r(S_j)$ on $M_j^{-1}(\Sigma_j\cap\bar L)\cap\{z: ze^{i\pi/4}> 0\}$ and $\bar R(S_j-)=-\frac{r(S_j)}{1-\left|r(S_j)\right|^2}$ on $M_j^{-1}(\Sigma_j\cap\bar L)\cap\{z: ze^{i\pi/4}< 0\}$.
\end{proposition}

\begin{remark}\label{Rem:1}
 For  even number $j$,  in the same way,  we can  get the following result
 \begin{align*}
 (\beta_jz+S_j)^{-2}N_jR(z)\delta_j^1(z)^2&=S_j^{-2}R(S_j\pm)e^{iz^2/2}z^{-2i\nu_j}\\+O(e^{-\frac{i}{2}\gamma z^2}t^{-\frac{1}{2}}\log t), &\ \ \ \text{ on }M_j^{-1}(\Sigma_j\cap L)\cap\{z:\pm ze^{i\pi/4}> 0\},\\
 N_jR(z)\delta_j^1(z)^2&=R(S_j\pm)e^{iz^2/2}z^{-2i\nu_j}\\+O(e^{-\frac{i}{2}\gamma z^2}t^{-\frac{1}{2}}\log t), &\ \ \ \text{  on }M_j^{-1}(\Sigma_j\cap L)\cap\{z:\pm ze^{i\pi/4}> 0\},\\
 \beta_jz+S_j)^{-2}N_j\bar R(z)\delta_j^1(z)^{-2}&=S_j^{-2}\bar R(S_j\pm)e^{iz^2/2}z^{-2i\nu_j}\\+O(e^{\frac{i}{2}\gamma z^2}t^{-\frac{1}{2}}\log t), &\ \ \ \text{ on }M_j^{-1}(\Sigma_j\cap L)\cap\{z:\pm z/e^{i\pi/4}> 0\},\\
 N_j\bar R(z)\delta_j^1(z)^{-2} &=\bar R(S_j\pm)e^{iz^2/2}z^{-2i\nu_j}\\+O(e^{\frac{i}{2}\gamma z^2}t^{-\frac{1}{2}}\log t), &\ \ \ \text{ on }M_j^{-1}(\Sigma_j\cap L)\cap\{z:\pm z/e^{i\pi/4}> 0\}.
 \end{align*}
\end{remark}

We consider some new matrices and operators that are limits of those on $\Sigma_j(0)$ which is  a union of four parts $$\Sigma_j(0)=\bigcup_{k=1}^k\Sigma_j^k(0),$$where $$\Sigma_j^k(0)=e^{i(2k-1)\pi/4}\mathbb{R}^+.$$

For $j=1, 2, 3, 4$, we define
\[
 w_\pm^{j,\infty}=\lim_{t\to\infty}\tilde w_\pm^j, \ \ \ w^{j,\infty}=w_+^{j,\infty}+w_-^{j,\infty}.
\]
Defining  operators
\begin{equation}
 \alpha_j^\infty=C_{w^{j,\infty}}: L^2(\Sigma_j(0))\to L^2(\Sigma_j(0)),
\end{equation}
then by (\ref{eq:105}), (\ref{eq:107}) and {Remark \ref{Rem:1}}, we have
$$\parallel \tilde w_\pm^j-w_\pm^{j,\infty}\parallel_{L^2}\leq Ct^{-1}\log t,$$
$$\parallel(C_{\tilde w^j}-C_{w^{j,\infty}})I\parallel_{L^2}\leq Ct^{-1}\log t.$$
Direct calculation shows that
\begin{equation}\label{eq:109}
\parallel (1-\alpha_j^\infty)^{-1}I-(1-\alpha_j)^{-1}I\parallel_{L^2}\leq Ct^{-1}\log t.
\end{equation}

In summary,   by using  (\ref{eq:56}), (\ref{eq:103}), (\ref{eq:104}), (\ref{eq:106}), (\ref{eq:109}),  Remark \ref{Rem:1}, we get that
\begin{align}\label{eq:77}
q_n=&\frac{\delta(0)^{-1}}{2\pi i}\sum_{j=1}^4\beta_jS_j^{-2}(\delta_j^0)^2\Big\{\int_{\Sigma_j(0)}[(1-\alpha_j^\infty)^{-1}I](z)w^{j,\infty}(z)dz\Big\}_{12}\notag\\
&+O(t^{-1}\log t).
\end{align}

For the branch cut of $z^{i\nu_j}$, we first consider that of $\delta_j^1$. From (\ref{eq:85}), it is reasonable to consider $M_j^{-1}\wideparen{S_jT_j}\cup( \beta_j^{-1}(T_j-S_j)+\mathbb{R}^-)$ as the branch cut of $\delta_j^1$. As $t\to\infty$, we find that this branch cut becomes $\mathbb{R}^-=\{x\in\mathbb{R}\}$. Therefore, we consider $\mathbb{R}^-$ as the branch cut of $z^{i\nu_j}$ in the remaining of this article.

\section{The boundedness of operators}\label{S:9}
\indent

In this section, we would investigate the existence and boundedness of  the following  operators
\begin{align}
&(I-C_{w^{(2)}})^{-1}, \  \  (I-C_{w'})^{-1}, \ (1-A_j)^{-1},\ \  (1-\hat A_j)^{-1}, \nonumber\\
& (1-\alpha_j)^{-1}, \ \  (1-\alpha_j^\infty)^{-1} \ \ ( j=1, 2, 3, 4).\nonumber
\end{align}

 \begin{remark}\label{re:9.1}
 	Assume that  operator $(1-\alpha_j^\infty)^{-1}$
 	exists and is  bounded on $L^2(\Sigma_j(0))$ for all $j$. With Proposition \ref{Pro:1}, \ref{Pro:2} and Remark \ref{Rem:1}, we deduce that
 	$ (1-\alpha_j)^{-1} $
 	exists and is uniformly bounded as $t\to\infty$.
 	Because $M_j$ is homeomorphic  from $\Sigma_j(0)$ to $\Sigma(S_j)$, $N_j$ is invertible and $\|N_j\|_{L^2(\Sigma(S_j)\to L^2(\Sigma_j(0))}$   is bounded as well as $\|N_j^{-1}\|_{L^2(\Sigma(S_j)\to L^2(\Sigma_j(0))}$. Moreover, noticing that $\Delta_j^0$ and its inverse are invertible, we get the existence and boundedness of $(1-\hat A_j)^{-1}$  from those of $ (1-\alpha_j)^{-1} $. More precisely, by equation (\ref{eq:70}), we have $$(1-\hat A_j)^{-1}=N_j^{-1}(\tilde\Delta_j^0)^{-1}(1-\alpha_j)^{-1}\tilde\Delta_j^0N_j.$$
 	
 	By (2.58) in \cite{deift1993steepest}, we derive the existence and boundedness of $(1-A_j)^{-1}$ from those of
	  $(1-\hat A_j)^{-1}$; thus, $(1-C_{w'}^{\Sigma'})^{-1}$ exists and is uniformly bounded  with equation (\ref{eq:51s}).
	  Moreover,   by (2.59) in \cite{deift1993steepest},  $(I-C_{w'})^{-1}$  exists and is uniformly bounded.
 	
 	Because of (\ref{eq:40})-(\ref{eq:41}) and boundedness of  $\Sigma^{(2)}$, we get that $$\|C_{w^{(2)}}-C_{w'}\|_{L^2(\Sigma^{(2)})\to L^2(\Sigma^{(2)})}\to 0.$$  Thus, by the second resolvent identity, we obtain that for sufficiently large t, the existence and boundedness of $(I-C_{w'})^{-1}$ implies those of $(I-C_{w^{(2)}})^{-1}$.
 \end{remark}

According to {Remark \ref{re:9.1}},   the remaining work in this section is to prove the existence and boundedness of $(1-\alpha_j^\infty)^{-1}$.  We discuss  this problem in  two cases when  $j$ is odd and $j$ is   even.

If  j is odd, we change the orientation of $\Sigma_j^1(0)$ and $\Sigma_j^4(0)$ in $\Sigma_j(0)$. In fact, this change doesn't affect $\alpha_j^\infty$ as an operator on $L^2(\Sigma_j(0))$. So  we could add the real line with orientation marked in {Figure 7},    we   define matrices
\begin{align*}
	w_\pm^{j,e}&=\begin{cases}
	w_\pm^{j,\infty} &\text{ on }\Sigma_j^2(0)\cup\Sigma_j^3(0),\\
	-w_\mp^{j,\infty} &\text{ on }\Sigma_j^1(0)\cup\Sigma_j^4(0),\\
	0 & \text{ on }\mathbb{R},
	\end{cases}\\
	w^{j,e}&=w_+^{j,e}+w_-^{j,e}, \ \  \alpha_j^e=C_{w^{j,e}},
\end{align*}
where
\begin{align*}
w_+^{j,e}=&
\begin{cases}
\left(
\begin{matrix}
0 & 0 \\
r(S_j)e^{iz^2/2}z^{-2i\nu_j} & 0
\end{matrix}\right) & z\in e^{i\pi/4}\mathbb{R}_+, \\
\left(
\begin{matrix}
0 & -\frac{\bar r(S_j)}{1-|r(S_j)|^2}e^{-iz^2/2}z^{2i\nu_j}\\
0&0
\end{matrix}\right) & z\in e^{3i\pi/4}\mathbb{R}_+, \\
0 & \text{otherwise  on $\Sigma^e$},
\end{cases} \\
w_-^{j,e}=&
\begin{cases}
\left(
\begin{matrix}
0 & -\bar r(S_j)e^{-iz^2/2}z^{2i\nu_j} \\
0 & 0
\end{matrix}\right) & z\in e^{-i\pi/4}\mathbb{R}_+, \\
\left(
\begin{matrix}
0 & 0 \\
\frac{r(S_j)}{1-|r(S_j)|^2}e^{iz^2/2}z^{-2i\nu_j} & 0
\end{matrix}\right) & z\in e^{-3i\pi/4}\mathbb{R}_+, \\
0 & \text{otherwise  on $\Sigma^e$}.
\end{cases}
\end{align*}
Then, by (2.58) in \cite{deift1993steepest}, we only need  estimate the bound of $(1-\alpha_j^e)^{-1}$.

 \ \ \ \ \ \ \ \ \ \ \ \ \ \ \ \ \ \
 \begin{tikzpicture}
 \pgfplotsset{
   every axis/.append style={
   extra description/.code={
   	\node at (0.5,0.1) { {Figure 7} $\Sigma_j^e$ for $j$ is odd};
   \node at (0.7,0.57){$\Omega_1^e$};
   \node at (0.5,0.68){$\Omega_2^e$};
   \node at (0.3,0.57){$\Omega_3^e$};
   \node at (0.3,0.43){$\Omega_4^e$};
   \node at (0.7,0.43){$\Omega_6^e$};
   \node at (0.5,0.32){$\Omega_5^e$};
   \node at (0.5,0.57){0};
   \filldraw [black] (0.5,0.5) circle (1.3pt);
   }
   }
   }
 \begin{axis}[
 hide x axis,
 hide y axis,
ymin=-4,
ymax=4,
xmin=-4.5,
xmax=4.5,
 ]
 \addplot[domain=-2:2]
 {x};
 \addplot[domain=-2:-1,->]
 {x};
 \addplot[domain=1:2,>-]
 {x};
 \addplot[domain=-2:2]
 {-x};
 \addplot[domain=-2:-1,->]
 {-x};
 \addplot[domain=1:2,>-]
 {-x};
 \addplot[domain=-2*sqrt(2):2*sqrt(2)]
 {0};
 \addplot[domain=-2*sqrt(2):-sqrt(2),-<]
 {0};
 \addplot[domain=sqrt(2):2*sqrt(2),<-]
 {0};
 \end{axis}
\end{tikzpicture}

 We write
 \[
 v^{j,e}=(b_-^{j,e})^{-1}b_+^{j,e}=(1-w_-^{j,e})^{-1}(1+w_+^{j,e}),
 \]
and define a meromorphic function  $\sigma$
\begin{equation}
 \sigma=
 \begin{cases}
  z^{i\nu_j\sigma_3} & z\in\Omega_2^e\cup\Omega_5^e,\\
  b_+^{j,e}z^{i\nu_j\sigma_3} & z\in\Omega_1^e\cup\Omega_3^e,\\
  b_-^{j,e}z^{i\nu_j\sigma_3} & z\in\Omega_4^e\cup\Omega_6^e.
 \end{cases}
\end{equation}
We simply denote  $v^e=v^{j,e}$,  also  $\det\sigma=1$  implies that   $\sigma$ is invertible. Let
\begin{equation}
 v^{e,\sigma}=\sigma_-^{-1}v^e\sigma_+, \ \ \text{ on }\Sigma_j^e,
\end{equation}
then     we find that
\begin{equation}
 v^{e,\sigma}=
 \begin{cases}
  v^0:=e^{-\frac{iz^2}{4}\hat\sigma_3}\left(
  \begin{matrix}
   1 & \bar r(S_j)\\
   -r(S_j) & 1-|r(S_j)|^2
  \end{matrix}\right) &\text{on $\mathbb{R}$},\\
  I & \text{on $\Sigma^e\setminus \mathbb{R}$},
 \end{cases}
\end{equation}
where
\begin{align}
 v^{e,\sigma}=&(b_-^{e,\sigma})^{-1}b_+^{e,\sigma}=(1-w_-^{e,\sigma})^{-1}(1-w_+^{e,\sigma})\notag\\
 =&\left(
 \begin{matrix}
 1 & 0\\
 -r(S_j)e^{iz^2/2} & 1
 \end{matrix}\right)\left(
 \begin{matrix}
   1 & \bar r(S_j)e^{-iz^2/2}\\
  0 & 1
 \end{matrix}\right).
\end{align}
Let
$$w^{e,\sigma}=w_+^{e,\sigma}+w_-^{e,\sigma}. $$
Since the operator $C_{w^{e,\sigma}|_{\mathbb{R}}}$: $L^2(\mathbb{R})\to L^2(\mathbb{R})$ satisfies
 $$\parallel C_{w^{e,\sigma}|_{\mathbb{R}}}\parallel\leq  \sup_{z\in\mathbb{R}}|e^{-iz^2/2}\bar r(S_j)|\leq \parallel r\parallel_{L^\infty(\mathbb{R})}< 1,$$
   by Lemma 2.56 in \cite{deift1993steepest}, we have  $\parallel C_{w^{e,\sigma}}\parallel< 1$ , where $C_{w^{e,\sigma}}$ is an operator on $L^2(\Sigma^e)$. Thus  $(1-C_{w^{e,\sigma}})^{-1}$ exists and is  bounded.
   Like  the step 5 in section 3 of \cite{deift1993steepest},  we can  deduce  the existence and  boundedness of $(1-\alpha_j^e)^{-1}$.  So the boundedness of $(1-\alpha_j^\infty)^{-1}$ follows.

As for $j$ is even, we could go through  the process once again as $j$ is odd,  but we   simply complete it by making  conjugate transformation to $w_\pm^{j,e}$.  Notice that for $j$ is even,  $\Sigma_j^e$ admits the orientation that is different from the case of $j$ is odd. Exactly, each line of $\Sigma_j^e$ admits different orientation.

Let  $w^{j,e}=w_+^{j,e}+w_-^{j,e}$,  where  $w_\pm^{j,e}$ are given by
\begin{equation*}
w_\pm^{j,e}=\begin{cases}
w_\pm^{j,\infty} &\text{ on }\Sigma_j^2(0)\cup\Sigma_j^3(0),\\
-w_\mp^{j,\infty} &\text{ on }\Sigma_j^1(0)\cup\Sigma_j^4(0),\\
0  & \text{ on }\mathbb{R}.
\end{cases}
\end{equation*}
Further,
\begin{align*}
 w_+^{j,e}=&
 \begin{cases}
  \left(
   \begin{matrix}
    0 & 0 \\
    r(S_j)e^{-iz^2/2}z^{2i\nu_j} & 0
   \end{matrix}\right) & z\in e^{-i\pi/4}\mathbb{R}_+, \\
   \left(
   \begin{matrix}
    0 & -\frac{\bar r(S_j)}{1-|r(S_j)|^2}e^{iz^2/2}z^{-2i\nu_j}\\
    0&0
   \end{matrix}\right) & z\in e^{-3i\pi/4}\mathbb{R}_+, \\
   0 & \text{otherwise  on $\Sigma^e$},
 \end{cases} \\
 w_-^{j,e}=&
 \begin{cases}
  \left(
   \begin{matrix}
    0 & -\bar r(S_j)e^{iz^2/2}z^{-2i\nu_j} \\
    0 & 0
   \end{matrix}\right) & z\in e^{i\pi/4}\mathbb{R}_+, \\
   \left(
   \begin{matrix}
    0 & 0 \\
    \frac{r(S_j)}{1-|r(S_j)|^2}e^{-iz^2/2}z^{2i\nu_j} & 0
   \end{matrix}\right) & z\in e^{3i\pi/4}\mathbb{R}_+, \\
   0 & \text{otherwise  on $\Sigma^e$}.
 \end{cases}
\end{align*}
Define a conjugate operator $T$:
$$Tf(z)=\overline{f(\bar z)}.$$ We could find that $Tw_\pm^{j,e}$ is almost the same as $w_\pm^{j,e}$ on $\Sigma^e$ by  replacing  $\overline{r(S_j)}$  to  $r(S_{j-1})$.  Noticing that
 $$\parallel\bar r\parallel_{L^\infty}=\parallel r\parallel_{L^\infty}<1,$$
the proof of   even number   $ j$  follows as    the proof of odd number  $ j $.

\section{The asymptotic   of discrete mKdV equation}
\indent

In this section, we give   the asymptotic behavior of $q_n$ shown in equation (\ref{eq:114}).

By  (\ref{eq:77}), we only have to estimate
\begin{equation}
\Big\{\int_{\Sigma_j(0)}[(1-\alpha_j^\infty)^{-1}I](z) w^{j,\infty}(z)dz\Big\}_{12}.
\end{equation}

We expand the  $m^j$ in the form
\begin{align}\label{eq:83}
m^j(z)&=I+\frac{1}{2\pi i}\int_{\Sigma_j(0)}\frac{[(1-\alpha_j^\infty)^{-1}I](\tau)w^{j,\infty}(\tau)d\tau}{\tau-z}\nonumber\\
&=I-z^{-1}m_1^j(z)+\cdots,
\end{align}
where
\begin{align}
m_1^j(z)=\int_{\Sigma_j(0)}[(1-\alpha_j^\infty)^{-1}I](z)w^{j,\infty}(z)\frac{dz}{2\pi i}.\label{pop}
\end{align}
And  $m^j$ is analytic on $\mathbb{C}\setminus\Sigma_j(0)$ and satisfy
\begin{equation}
m_+^j=m_-^jv^{j,\infty}\ \  \text{on $\Sigma_j(0)$},
\end{equation}
where $$ v^{j,\infty}=(1-w_-^{j,\infty})^{-1}(1+w_+^{j,\infty}).$$

For   odd number $j$,   we set  $$\Sigma_j^\infty=\Sigma_j^e,$$  and the  orientation of $\Sigma^\infty$ is different from $\Sigma^e$ only on $\mathbb{R}$ (See the Figure 8).  Further  we  denote
 $\tilde w_\pm^{j,\infty}$  as the zero extension of $w_\pm^{j,\infty}$ on $\Sigma_j^\infty$.

 \ \ \ \ \ \ \ \ \ \ \ \ \ \ \ \ \ \
\begin{tikzpicture}
\pgfplotsset{
	every axis/.append style={
		extra description/.code={
			\node at (0.5,0.1) { {Figure 8}. \emph{$\Sigma_j^\infty$ for $j$ is odd}};
			\node at (0.7,0.57){$\Omega_1^e$};
			\node at (0.5,0.68){$\Omega_2^e$};
			\node at (0.3,0.57){$\Omega_3^e$};
			\node at (0.3,0.43){$\Omega_4^e$};
			\node at (0.7,0.43){$\Omega_6^e$};
			\node at (0.5,0.32){$\Omega_5^e$};
			\node at (0.5,0.57){0};
			\filldraw [black] (0.5,0.5) circle (1.3pt);
		}
	}
}
\begin{axis}[
hide x axis,
hide y axis,
ymin=-4,
ymax=4,
xmin=-4.5,
xmax=4.5,
]
 \addplot[domain=-2:2]
{x};
\addplot[domain=-2:-1,->]
{x};
\addplot[domain=1:2,>-]
{x};
\addplot[domain=-2:2]
{-x};
\addplot[domain=-2:-1,->]
{-x};
\addplot[domain=1:2,>-]
{-x};
\addplot[domain=-2*sqrt(2):2*sqrt(2)]
{0};
\addplot[domain=-2*sqrt(2):-sqrt(2),->]
{0};
\addplot[domain=sqrt(2):2*sqrt(2),>-]
{0};
\end{axis}
\end{tikzpicture}

Define  $\Phi(z)=m^j(z)\sigma(z)$ on $\mathbb{C}\setminus\Sigma^\infty$, then  we would find that
\begin{equation}
 \Phi_+=
 \begin{cases}\label{eq:86}
  \Phi_- & \text{on $  \ (e^{i\pi/4}\mathbb{R})\cup(e^{-i\pi/4}\mathbb{R})\setminus(0)$},\\
  \Phi_-(v^0)^{-1} & \text{on \  $\mathbb{R}$}.
 \end{cases}
\end{equation}
From  the asymptotic behavior of $m^j$ as $z\to\infty$, we have
\begin{equation}\label{eq:87}
 \Phi(z)z^{-i\nu_j\sigma_3}=I-z^{-1}m_1^j+\cdots.
\end{equation}
which implies that $\Phi$ is analytic on $\mathbb{C}\setminus\mathbb{R}$.  Letting  $\hat\Phi=\Phi z^{-\frac{iz^2}{4}\sigma_3}$ on $\mathbb{C}\setminus\mathbb{R}$,   we then  get
a model  RHP
\begin{align*}
\begin{cases}
\hat\Phi_+=\hat\Phi_-\left(
\begin{matrix}
1-|r(S_j)|^2 & -\overline{r(S_j)}\\
r(S_j) & 1
\end{matrix}\right) \text{ on }\mathbb{R},\\
\hat\Phi z^{-i\nu_j\sigma_3}e^{\frac{iz^2}{4}\sigma_3}\to I \text{ as }z\to\infty.
\end{cases}
\end{align*}
By using the result  (110) in  \cite{Deift1994conf},  we could get that
\begin{equation}
 (m_1^j)_{12}=\frac{i(2\pi)^{1/2}e^{i\pi/4}e^{-\pi\nu_j/2}}{r(S_j)\Gamma(-i\nu_j)},\  \ j=1, 3. \label{pop1}
\end{equation}

For  the case  when  $j$ is even,   we can write $m^j$   as
\begin{align}
 m^j&=I+\int_{\Sigma_j(0)}\frac{[(1-C_{w^{j,e}})^{-1}I](\tau)w^{j,e}(\tau)d\tau}{\tau-z}\notag\\&=I-z^{-1}m_1^j(z)+\cdots,
\end{align}
where
\begin{equation}
 m_1^j=\int_{\Sigma_j(0)}[(1-C_{w^{j,e}})^{-1}I](z)w^{j,e}(z)dz.
\end{equation}
Because $C_{w^{j,e}}=TC_{Tw^{j,e}}T$ and $T^2=1$, we obtain
\begin{equation}
 C_{Tw^{j,e}}=TC_{w^{j,e}}T
\end{equation}
and
\begin{align}\label{eq:92}
  &\int_{\Sigma_j(0)}[(1-C_{Tw^{j,e}})^{-1}I](z)Tw^{j,e}(z)dz\notag\\
  =&\int_{\Sigma_j(0)}[T(1-C_{w^{j,e}})^{-1}TI](z)Tw^{j,e}(z)dz\notag\\
  =&\int_{\Sigma_j(0)}T[(1-C_{w^{j,e}})^{-1}I\times w^{j,e}](z)dz\notag\\
  =&\overline{\int_{\Sigma_j(0)}[(1-C_{w^{j,e}})^{-1}I](z)w^{j,e}(z)dz}.
\end{align}
Therefore,   in a similar way to the case when $j$ is odd,   we have
\begin{align}\label{eq:93}
  &\Big\{\int_{\Sigma_j(0)}[(1-C_{Tw^{j,e}})^{-1}I](z)Tw^{j,e}(z)dz\Big\}_{12}=\frac{i(2\pi)^{1/2}e^{i\pi/4}e^{-\pi\nu_j/2}}{\overline{r(S_j)}\Gamma(-i\nu_j)}.
\end{align}
Combining (\ref{eq:92}) and (\ref{eq:93}) gives
\begin{equation}
 (m_1^j)_{12}=-\frac{i(2\pi)^{1/2}e^{-i\pi/4}e^{-\pi\nu_j/2}}{r(S_j)\Gamma(i\nu_j)}, \ \ j=2, 4, \label{pop2}
\end{equation}
where  $\nu_j=-\frac{1}{2\pi}\log(1-|r(S_j)|^2)$.

Finally, combining (\ref{eq:77}), (\ref{eq:83}), (\ref{pop}), (\ref{pop1})  and (\ref{pop2}) gives  the  asymptotic behavior of the  discrete  mKdV equation
\begin{equation}\label{eq:114}
  q_n=\delta(0)^{-1}\sum_{j=1}^4\beta_jS_j^{-2}(\delta_j^0)^2(m_1^j)_{12} +O(t^{-1}\log t), \ \  \text{for} \   |n|\leq V_0t,
\end{equation}
where $S_j, \ j=1, 2, 3, 4$ are stationary points;  $\beta_j$ and  $\delta_j^0$are given by (\ref{pop5}) and (\ref{eq:96}) respectively, and
\begin{equation}
 (m_1^j)_{12}=
 (-1)^{j-1}\frac{i(2\pi)^{1/2}e^{-i\pi/4}e^{-\pi\nu_j/2}}{r(S_j)\Gamma[(-1)^ji\nu_j]},\ \ \ j=1, 2, 3, 4.\notag
\end{equation}

The  asymptotic formula  (\ref{eq:114}) is composed  of  a  decaying term and  a  leading term   coming from  four stationary phase points $S_j, \ j=1, 2, 3, 4$.
In the leading term  $\delta^0_j$    contains three oscillatory factors: $S_j^n$, $e^{-\frac{t}{2}(S_j^2-S_j^{-2})}$ and $\beta_j^{(-1)^{j-1}i\nu_j}$.
Since $n/t$ is fixed, as $t$ tends to the infinity, $n$  also tends to the infinity.
If we set $\theta_j=\arg S_j$ and $\kappa_j$ the imaginary  part of $ S_j^2$, then we can write the three oscillatory terms in the form
\begin{align*}
&S_j^ne^{-\frac{t}{2}(S_j^2-S_j^{-2})}\beta_j^{(-1)^{j-1}i\nu_j}\notag = e^{\frac{i}{2}(2n\theta_j-2\kappa_j t+(-1)^j\nu_j\log t)}\psi_j(n/t),
\end{align*}
where $\psi_j$ is a function about $n/t$. Then, $\beta_j(\delta_j^0)^2$ behaves like const.$t^{-1/2}e^{ip_jt+iq_j\log t}$, $p_j\in \mathbb{R}$ and $q_j\in\mathbb{R}$.

In this article,   we   have got    the long time asymptotic  formula
  (\ref{eq:114}) for  the  solutions of  the  initial  value problem   for   the   discrete defocusing   mKdV equation   (\ref{eq:1})-(\ref{eq:2})
 by the Deift-Zhou steepest descent method.
  To our knowledge,  with exception to   recent work on Toda lattice   and discrete Schr\"{o}dinger equation   \cite{kruger2009long,yamane2014long,yamane2015,yamane20191,yamane20192},  there has been little  work on asymptotic behavior for didscrete integrable
  systems via the Deift-Zhou steepest descent method.
  There are almost no work  on asymptotic behavior for   discrete  integrable systems  with  nonzero boundary conditions.  \vspace{4mm}

{\bf Acknowledgements} Many thanks for  referees for helpful suggestions in improving the manuscript.
 This work   was supported by the National Science Foundation of China under Project
  No.11671095 and  No.51879045.

\end{document}